\pdfmapfile{=<name>.map}
\documentclass[12pt,oneside,reqno,fleqn]{amsart}
\usepackage{amssymb}
\usepackage{bbm}
\usepackage{cases}
\usepackage{amsmath,mathtools}
\usepackage{graphicx}
\usepackage{mathrsfs}
\usepackage{stmaryrd}
\usepackage{color}
\usepackage{soul}
\usepackage[dvipsnames]{xcolor}
\usepackage{amsfonts}
\usepackage{amsmath,amssymb,amsthm}
\usepackage{libertine}
\usepackage[T1]{fontenc}
\usepackage[libertine,varbb]{newtxmath}
\usepackage[colorlinks,linkcolor=Red,citecolor=Red,urlcolor=Red]{hyperref}
\usepackage[shortlabels] {enumitem}
\usepackage[nameinlink,capitalize]{cleveref}
\usepackage[spacing=true,kerning=true,tracking=true,letterspace=50]{microtype}
\usepackage{orcidlink}
\usepackage{ulem}
\usepackage[disable]{todonotes}
\hypersetup{pdftitle={Kinetic SDEs}, pdfauthor={Chengcheng Ling}}
\numberwithin{equation}{section}
\newcommand{\be}{\begin{eqnarray}}
\newcommand{\ee}{\end{eqnarray}}
\newcommand{\ce}{\begin{eqnarray*}}
\newcommand{\de}{\end{eqnarray*}}
\newtheorem{theorem}{Theorem}[section]
\newtheorem{lemma}[theorem]{Lemma}
\newtheorem{proposition}[theorem]{Proposition}
\newtheorem{corollary}[theorem]{Corollary}
\theoremstyle{remark}
\newtheorem{assumption}[theorem]{Assumption}

\newtheorem{example}[theorem]{Example}
\newtheorem{remark}[theorem]{Remark}
\newtheorem{definition}[theorem]{Definition}

\crefname{eqn}{Equation}{Equations}
\crefname{assumption}{Assumption}{Assumptions}
\crefname{innercustomthm}{Condition}{Conditions}
\crefrangelabelformat{innercustomthm}{#3#1#4-#5#2#6}

\def\bbp{{\boldsymbol{p}}}

\def\bbq{{\boldsymbol{q}}}
\def\bba{{\boldsymbol{a}}}

\def\bb2{{\boldsymbol{2}}}

\def\={&\!\!=\!\!&}

\def\e{{\mathrm{e}}}
\def\eps{\varepsilon}

\def\b{\beta}
\def\p{\partial}

\def\<{{\langle}}
\def\>{{\rangle}}
\def\({{\Big(}}
\def\){{\Big)}}

\def\bx{{\mathbf{x}}}

\def\dif{d}

\def\={&\!\!=\!\!&}
\def\bt{\begin{theorem}}
\def\et{\end{theorem}}
\def\bl{\begin{lemma}}
\def\el{\end{lemma}}
\def\br{\begin{remark}}
\def\er{\end{remark}}
\def\bd{\begin{definition}}
\def\ed{\end{definition}}
\def\bp{\begin{proposition}}
\def\ep{\end{proposition}}
\def\bc{\begin{corollary}}
\def\ec{\end{corollary}}
\def\bx{\begin{example}}
\def\ex{\end{example}}
\def\cA{{\mathcal A}}

\def\cF{{\mathcal F}}

\def\cG{{\mathcal G}}

\def\cI{{\mathcal I}}

\def\cR{{\mathcal R}}
\def\cS{{\mathcal S}}

\def\mC{{\mathbb C}}

\def\mE{{\mathbb E}}
\def\E{\mE}

\def\mL{{\mathbb L}}

\def\mN{{\mathbb N}}

\def\mP{{\mathbb P}}

\def\mR{{\mathbb R}}

\def\sC{{\mathscr C}}

\def\sF{{\mathscr F}}
\def\sG{{\mathscr G}}

\def\bC{{\mathbb C}}
\def\bB{{\mathbb B}}

\def\geq{\geqslant}
\def\leq{\leqslant}

\newcommand{\R}{{\mathbb R}}

\newcommand{\norm}[1]{{\left\vert\kern-0.25ex\left\vert\kern-0.25ex\left\vert #1
    \right\vert\kern-0.25ex\right\vert\kern-0.25ex\right\vert}}

\renewcommand{\le}{\leq}
\renewcommand{\ge}{\geq}
\allowdisplaybreaks
\begin{document}
		\title{Strong convergence of the Euler scheme for singular kinetic SDEs driven by $\alpha$-stable processes}
 \author{ Chengcheng Ling 
}
	\date{\today}
 \address{ University of Augsburg, Institut f\"ur Mathematik,
86159 Augsburg,  Germany.}
\email{ chengcheng.ling@uni-a.de}
\begin{abstract}
We study the strong approximation of the solutions to singular stochastic kinetic equations (also referred to as second-order SDEs) driven by $\alpha$-stable processes, using an Euler-type scheme inspired by \cite{HLL}. For these equations, the stability index $\alpha$ lies in the range $(1,2)$, and the drift term exhibits anisotropic $\beta$-Hölder continuity with $\beta >1 - \frac{\alpha}{2}$. We establish a convergence rate of $(\frac{1}{2} + \frac{\beta}{\alpha(1+\alpha)} \wedge \frac{1}{2})$, which aligns with the results in \cite{BDG-Levy} concerning first-order SDEs.
	\bigskip
		
  \noindent {{\sc Mathematics Subject Classification (2020):}
		Primary 60H35, 
  65C30, 
  60H10; 
		Secondary
		60H50, 
		60L90. 
		}

		\noindent{{\sc Keywords:} Singular SDEs; 
  Strong approximation; Kinetic SDEs; Second order SDEs; Euler scheme; Regularization by noise; Stochastic sewing, $\alpha$-stable noise.}
	\end{abstract}
	
	\maketitle

\section{Introduction}
We consider the following SDE:
\begin{align}\label{eq:SDE}
\begin{cases}
\dif X_t=V_t\dif t,\\
\dif V_t=b(Z_t)\dif t+\dif L_t,\quad Z_0=\zeta=(\xi,\eta)\in\cF_0,
\end{cases}
\end{align}
where $Z_t=(X_t,V_t)\in\mR^d\times\mR^d$, $b:\mR^{2d}\to\mR^d$ is measurable, and $L$ is a standard $d$-dimensional $\alpha$-stable process defined on the probability space $(\Omega,\cF,(\cF_s)_{s\ge0},\mP)$ where $\alpha\in(1,2)$; that is to say, $L$ is a L\'evy process with characteristic function $\Phi$ (i.e.$\forall t\geq0$, $\forall \xi\in\mR^d$, $\mE(\exp(i\xi\cdot L_t))=\exp(-t\Phi(\xi))$, $i$  is the complex unit) so that
\begin{align}\label{def:Phi}
    \Phi(\xi)=\int_{\mR^d\backslash\{0\}}(1-e^{i\xi\cdot x}+i\xi\cdot x\mathbf{1}_{|x|\leq1})\frac{dx}{|x|^{d+\alpha}}.
\end{align}
$\alpha$ is the stable index of $L$ which in general takes value in $(0,2]$. When $\alpha=2$, $L$ corresponds to Brownian motion.

In classical mechanics, kinetic equations describe the motion of particles with momentum and inertia.  In many physical and natural systems—such as  particles moving in disordered media—the environment is not smooth but rather influenced by irregular forces. In such cases, random "kicks" can occur, leading to non-Gaussian random perturbations \cite{C}.  \eqref{eq:SDE} represents such mathematical
models that describe the random evolution of physical systems with both inertial and random effects, often in environments with irregular, impulsive disturbances.   These equations usually are also called second-order SDEs because they account not only for the position of a particle but also its velocity and acceleration, thus incorporating both kinetic and stochastic dynamics (\cite{V, HRZ}); naturally the following classical SDE is called first order SDE:
\begin{align}
    \label{eq:non-de}dX_t=b(X_t)dt+dL_t,\quad X_0\in\mR^d.
\end{align}
In this paper we study \eqref{eq:SDE} with the focus on $b$ beyond Lipschitz continuous scope.  Reading through related literature, starting with \eqref{eq:non-de} driven by Brownian noise, there is a long list of results studying the equations with singular drifts \cite{Zv, Ver, KR, Zhang2011, G, GM}, and later with stable noise \cite{TTW, P, Csz, ABM, CZZ, LZ1, XZ, KP}, known as {\it{regularization by noise}} effect.

After obtaining the existence and uniqueness of the solutions to such equations,  developing accurate approximation schemes  is essential for understanding real-world systems where analytical solutions are rarely available. Approximation allows us to simulate the system  and gain insights into the statistical properties of its solutions, convergence rates. In particular, capturing the convergence rate is crucial for quantifying how well the numerical scheme represents the true dynamics.

For singular first order equation \eqref{eq:non-de}, again there is a long list of work studying
 its numerical approximations \cite{BDG, DGL, LL, MGY20, PSS, PS, PSS1}, which
indicates the well understanding for the first order singular systems. For second order equation  in the form of \eqref{eq:SDE}, there are fewer results. There are results on studying the solutions and its numerics in at least-Lipschitz regular setting, see \cite{PM, SW}.   Here we list those that we can find in the literature which however mostly are about Brownian case. \cite{MSH}  and \cite{T} study the time-discrete approximations (for instance, explicit and implicit Euler-Maruyama scheme) for \eqref{eq:SDE} under Lyapunov function type conditions which usually require certain regularity  on $b$. Later  \cite{LM}  obtains the non asymptotic bounds for the Monte Carlo algorithm.  \cite{LS18} seems to be the first one obtaining the quantitative strong convergence bound on the EM scheme with order $n^{-\frac{1}{4}}$ for $b$ being piecewise Lipschitz function which is different from the singularity type that we are considering here.

It is evident to see that the degenerate noise (i.e. vanishing noise in $X$ direction) and singularity of drift being together make the approximation problem more challenging.  For \eqref{eq:SDE}, \cite{HWZ20} shows the strong  well-posedness of it for $\alpha\in(1,2)$ and $b$ being $\gamma$-H\"older continuous in $x$ variable and $\beta$-H\"older continuous in $v$ with $\gamma\in (\frac{2+\alpha}{2(1+\alpha)},1)$ and $\beta\in (1-\frac{\alpha}{2},1)$. In  \cite{HRZ} the condition of $b$ for strong well-posedness is extended to anistrophic $\beta$-H\"older continuity with $\beta>1-\frac{\alpha}{2}$ which is also the condition that we assume in current work since we focus on its strong solutions;   however  for weak  well-posedness $b$ needs to be anistrophicly $\beta$-H\"older continuous only \cite{HRZ}.

Inspired by \cite{HLL},  we propose the following \emph{Euler scheme} to approximate \eqref{eq:SDE}: for any $n\in\mN$, define $k_n(t):=\frac{\lfloor nt \rfloor}{n}, t\in \mR_+$;
let $\{Z^n_t\}_{t\ge0}:=\{(X^n_t,V^n_t)\}_{t\ge0}$ be the solution to the following Euler scheme
\begin{align}\label{eq:SDE-EM}
\begin{cases}
X^n_t=\xi+\int_0^t V^n_s\dif s,\\
V^n_t=\eta+\int_0^t \Gamma_{s-k_n(s)}b(Z^n_{k_n(s)})\dif s+L_t,
\end{cases}
\end{align}
where
\begin{align*}
\Gamma_tf(z):=f(\Gamma_t z):=f(x+tv,v),\quad z=(x,v)\in\mR^d\times\mR^d.
\end{align*}
We may notice that comparing with the classical Euler-Maruyama scheme, there is no descrization on $X$ in algorithm \eqref{eq:SDE-EM} and it also contains an extra operator $\Gamma$ on $V$. Actually the appearance of $\Gamma$ does not cause any computational cost since it is simply a shift in $V$. The descrization of $X$ seems to more subtle for technical reason.   Nevertheless, noticing that $X^n$ is Lebesgue integral of $V$ which is rather explicit via its Euler-Maruyama scheme, it is still easier to be simulated comparing with \eqref{eq:SDE}.\\

Therefore, the goal of this paper is to show the convergence of $Z^n$ from \eqref{eq:SDE-EM} to $Z$ satisfying \eqref{eq:SDE}.  We mainly apply the {\it regularization by noise} techniques (\cite{LL, GL, BDG, DGL, GLL}) including {\it stochastic sewing lemma} \cite{Le} and {\it John-Nirenberg inequality} \cite{Le2022}. Usually these techniques also include {\it Zvonkin transformation} which is  one of the main tools for studying singular SDEs driven by L\'evy processes (\cite{Zv, Ver, Zhang2018, CMPZ}) and heavily replies on the regularity estimates of its corresponding Kolmogrov equation (\cite{B, IS, RZ24}). Differently from the aforementioned methods, we do not apply {\it Zvonkin transformation} but instead our idea is inspired from and closer to \cite{BDG-Levy}. Based on  the estimates on the transition density and semigroup of the process $M_t:=\left(\begin{smallmatrix}\int_0^t L_r dr\\L_t
    \end{smallmatrix}\right)$, we mainly show the estimates of the following types
\begin{align*}
S_1:=\| \int_S^T &f(M_r+\phi_r)-f(M_r+\varphi_r)dr\|_{L_p(\Omega)|\sG},
 \\
 S_2:=\| \int_S^T &f(M_r+\phi_r)-f(M_r+\psi_r)dr\|_{L_p(\Omega)|\sG}, \\
  S_3:= \| \int_S^T &f(M_r+\Gamma_{r-k_n(r)}\phi_{k_n(r)})-\Gamma_{r-k_n(r)}f(M_{k_n(r)}+\phi_{k_n(r)})dr\|_{L_2(\Omega)|\sG},
\end{align*}
where $\sG\subset\cF_s$, $\phi=Z^n-M, \varphi:=Z-M, \psi_t:=\Gamma_{t-k_n(t)}(Z_{k_n(t)}^n-M_{k_n(t)}),$ and $\|\cdot\|_{L_p(\Omega)|\sG}:=(\mE[|\cdot|^p|\sG])^\frac{1}{p}$, $p\geq1$.
Then the error of $Z-Z^n$ is obtained from these three type estimates above.  In the end we obtain the strong  convergence of $Z-Z^n$ with rate $n^{-(\frac{1}{2}+\frac{\beta}{\alpha(1+\alpha)})+\epsilon}$ where $\beta$ is the anisotropic H\"older  continuity of $b$ and $\epsilon$ is some sufficiently small constant.  The convergence rate also matches the rate obtained from \cite{BDG-Levy} on the first order SDE  \eqref{eq:non-de} and its Euler scheme
\begin{align}
    \label{eq:non-de-E}
  dX_t^n=b(X_{k_n(t)}^n)dt+dL_t,\quad X_0^n\in\mR^d.
\end{align}
\subsection*{Organization of the paper}
We introduce the frequently used notations and main result in Section \ref{sec:Notation-result}. The auxiliary estimates on kinetic semigroup and moments bound on the driven kinetic noise are collected in Section \ref{sec:auxi-tool}. Section \ref{sec:app-term} shows the convergence rate for the terms appearing in the difference $Z-Z^n$ and the final proof of the main result Section \ref{thm:main-rough} is presented in Section \ref{sec:fin-proof}.
\section{Notations and main results}\label{sec:Notation-result}
\subsection{Notations}
Here we collect all of the notations that we quite often use in the whole text. We start with introducing the anisotropic Besov spaces.
\subsection*{Anisotropic Besov spaces} \label{sec.AS}

For $f\in L^1(\mR^{2d})$, let $\hat f$ be the Fourier transform of $f$ defined by
$$
\hat f(\xi):=(2\pi)^{-2d}\int_{\mR^{2d}} \e^{-{\rm i}\xi\cdot z}f(z)\dif z, \quad\xi\in\mR^{2d},
$$
and $\check f$ denotes the inverse Fourier   transform of $f$ defined by
$$
\check f(z):=(2\pi)^{-2d}\int_{\mR^{2d}} \e^{{\rm i}\xi\cdot z}f(\xi)\dif\xi, \quad z\in\mR^{2d}.
$$
Let $\bba=(1+\alpha,1)$. For $z=(x,v)$, $z'=(x',v')\in\mR^{2d}$, we introduce the anisotropic distance
$$
|z-z'|_\bba:=|x- x'|^{1/(1+\alpha)}+|v-v'|.
$$
Note that $z\mapsto|z|_{\bba}$ is not smooth.
For $r>0$ and $z\in\mR^{2d}$, we also introduce the ball centered at $z$ and with radius $r$ with respect to the above distance
as follows:
$$
B^\bba_r(z):=\{z'\in\mR^{2d}:|z'-z|_\bba\leq r\},\ \ B^\bba_r:=B^\bba_r(0).
$$
Let $\chi^\bba_0$ be  a symmetric $C^{\infty}$-function  on $\mR^{2d}$ with
$$
\chi^\bba_0(\xi)=1\ \mathrm{for}\ \xi\in B^\bba_1\ \mathrm{and}\ \chi^\bba_0(\xi)=0\ \mathrm{for}\ \xi\notin B^\bba_2.
$$
We define
$$
\phi^\bba_j(\xi):=
\left\{
\begin{aligned}
&\chi^\bba_0(2^{-j\bba}\xi)-\chi^\bba_0(2^{-(j-1)\bba}\xi),\ \ &j\geq 1,\\
&\chi^\bba_0(\xi),\ \ &j=0,
\end{aligned}
\right.
$$
for $s\in\mR$ and $\xi=(\xi_1,\xi_2)$,
$
2^{s\bba }\xi=(2^{(1+\alpha)s}\xi_1, 2^{s}\xi_2).
$
Note that
\begin{align*}
{\rm supp}(\phi^\bba_j)\subset\big\{\xi: 2^{j-1}\leq|\xi|_\bba\leq 2^{j+1}\big\},\ j\geq 1,\ {\rm supp}(\phi^\bba_0)\subset B^\bba_2,
\end{align*}
and $\sum_{j\geq 0}\phi^\bba_j(\xi)=1,\ \ \forall\xi\in\mR^{2d}.$

Let $\cS$ be the space of all Schwartz functions on $\mR^{2d}$ and $\cS'$ the dual space of $\cS$, called the tempered distribution space.
For given $j\geq 0$, the  dyadic anisotropic block operator  $\mathcal{R}^\bba_j$ is defined on $\cS'$ as
\begin{align}\label{Ph0}
\mathcal{R}^\bba_jf(z):=(\phi^\bba_j\hat{f})\check{\ }(z)=\check{\phi}^\bba_j*f(z),
\end{align}
where the convolution is understood in the distributional sense and by scaling, one has
\begin{align}\label{SX4}
\check{\phi}^\bba_j(z)=2^{(j-1)(2+\alpha)d}\check{\phi}^\bba_1(2^{(j-1)\bba}z),\ \ j\geq 1.
\end{align}
We  define the isotropic block operator $\cR_jf=\check\phi_j*f$ in $\mR^d$, where
\begin{align}\label{Cx9}
{\rm supp}(\phi_j)\subset\big\{\xi: 2^{j-1}\leq|\xi|\leq 2^{j+1}\big\},\ j\geq 1,\ {\rm supp}(\phi_0)\subset B_2.
\end{align}
Now we introduce the following anisotropic Besov spaces (cf. \cite[Chapter 5]{Tri06}).
\begin{definition}\label{bs}
Let $s\in\mR$, $q\in[1,\infty]$ and $\bbp\in[1,\infty]^2$. The  anisotropic Besov space is defined by
$$
\mathbf{B}^{s,q}_{\bbp;\bba}:=\left\{f\in \cS': \|f\|_{\mathbf{B}^{s,q}_{\bbp;\bba}}
:= \left(\sum_{j\geq0}\big(2^{ js}\|\cR^\bba_{j}f\|_{\bbp}\big)^q\right)^{1/q}<\infty\right\},
$$
where $\|\cdot\|_\bbp$ is defined by $\|f\|_{\mL^\bbp}:=\|f\|_{\mL_z^\bbp}:=\|f\|_{\bbp}:=\left(\int_{\mR^d}\|f(\cdot,v)\|_{p_x}^{p_v}\dif v\right)^{1/p_v}.$\\
 Moreover,  $\bB^{s,q}_{p}$ denotes the usual  isotropic Besov spaces
in $\mR^d$ in terms of isotropic block operators $\cR_j$. If there is no confusion, we shall write
\begin{align}
    \label{def:Bes-pq}\bB^s_{\bbp;\bba}:=\bB^{s,\infty}_{\bbp;\bba},\ \ \bB^s_{p}:=\bB^{s,\infty}_{p},\quad \bC^s:=\bB^{s,\infty}_{\infty}.
\end{align}
\end{definition}
\begin{definition}\label{def:bs-Ho}
For a function $f:\mR^{2d}\to\mR$,  the first-order difference operator is defined by
$$
\delta^{(1)}_hf(z):=\delta_hf(z):=f(z+h)-f(z),\ \ z, h\in\mR^{2d}.
$$
For $M\in\mN$, the $M$-order difference operator  is defined recursively by
$
\delta^{(M+1)}_hf(z)=\delta_h\circ\delta^{(M)}_hf(z).
$
\end{definition}
The relation among the introduced Besov spaces and differential operator can be illustrated via the following  known results (cf. \cite{ZZ21} and \cite[Theorem 2.7]{HZZZ22}).
\begin{proposition}
   For $s>0$, $q\in[1,\infty]$ and $\bbp\in[1,\infty]^2$, an equivalent norm of $\bB^{s,q}_{\bbp;\bba}$ is given by
\begin{align}\label{CH1}
\|f\|_{\bB^{s,q}_{\bbp;\bba}}\asymp \left(\int_{\mR^{2d}}\left(\frac{\|\delta_h^{([s]+1)}f\|_{\bbp}}{|h|^s_\bba}\right)^q\frac{\dif h}{|h|^{4
d}_\bba}\right)^{1/q}+\|f\|_{\bbp}=:[f]_{\bB^{s,q}_{\bbp;\bba}}+\|f\|_{\bbp},
\end{align}
where $[s]$ is the integer part of $s$. In particular, $\bC^s_{\bba}:=\mathbf{B}^{s,\infty}_{\infty;\bba}$ is the anisotropic
H\"older-Zygmund space, and for $s\in(0,1)$, there is a constant $C=C(\alpha,d,s)>0$ such that
\begin{align*}
  \|f\|_{\bC^s_{\bba}} \asymp_C\|f\|_\infty+\sup_{z\not= z'}|f(z)-f(z')|/|z-z'|^{s}_\bba,
\end{align*}
\end{proposition}

\subsection*{Some conventions}
On finite dimensional vector spaces we always use the Euclidean norm.

Without confusion we sometimes also use $<,=, \leq$ and  $>,\geq$ between vectors if their each component shares the same order, e.g. we say $\bbq:=(q_i)_{i\geq 1}\leq \bbp:=(p_i)_{i\geq 1}$ if $q_i\leq p_i$ for each $i$.

In proofs, the notation $a\lesssim b$ abbreviates the existence of $C>0$ such that $a\leq C b$, such that moreover $C$ depends only on the parameters claimed in the corresponding statement. If the constant depends on any further parameter $c$, we incorporate it in the notation by writing $a\lesssim_c b$.
\subsection{Main results}
Our main assumption and result can be stated as follows.
\begin{assumption}
    \label{ass:main1}
 $b\in\bC_x^{\frac{\alpha+\beta}{1+\alpha}}\cap\bC_\bba^\beta$ with $\beta\in(1-\frac{\alpha}{2},(\alpha-1)(1+\alpha))$, $\alpha\in(1,2)$.
\end{assumption}
As a direct application of \cite[Theorem 4.1 (ii)]{HRZ} we know that under Assumption \ref{ass:main1} there exists a unique strong solution  $(Z_t)_{t\in[0,1]}$  to \eqref{eq:SDE}. It  can be approximated by $(Z^n_t)_{t\in[0,1]}$  from  \eqref{eq:SDE-EM} in the following sense.
\begin{theorem}
    \label{thm:main-rough}
        Let $(Z_t)_{t\in[0,1]}$ and $(Z^n_t)_{t\in[0,1]}$  be the solutions corresponding to \eqref{eq:SDE} and \eqref{eq:SDE-EM}.
       If   Assumption \ref{ass:main1}  holds, 
         then  we have  for any $m\in(2,\infty)$ and  $\eps>0$ being small, there exists a constant $N=N(\|b\|_{\mC_\bba^\beta}, \|b\|_{\mC_x^{\frac{\alpha+\beta}{1+\alpha}}}, m,\alpha,\beta,\epsilon)$ so that
         \begin{align}
             \label{est:thm-EM-s}
             \big\|  \sup_{t\in[0,1]}|Z_t-Z_t^n|\big\|_{L_m(\Omega)}\leq Nn^{-(\frac{1}{2}+\frac{\beta}{\alpha(1+\alpha)}\wedge \frac{1}{2})+\epsilon}. 
         \end{align}
\end{theorem}
In the similar manner of \cite{BDG-Levy} we also have the following almost sure convergence.
\begin{corollary}\label{cor:a.s}
    Let $(Z_t)_{t\in[0,1]}$ and $(Z^n_t)_{t\in[0,1]}$  be the solutions corresponding to \eqref{eq:SDE} and \eqref{eq:SDE-EM}.
       If  Assumption \ref{ass:main1} holds, 
         then  we have  for any $m\in(2,\infty)$ and  $\eps>0$ being small, there exists a constant $N=N(\|b\|_{\mC_\bba^\beta}, \|b\|_{\mC_x^{\frac{\alpha+\beta}{1+\alpha}}}, m,\alpha,\beta,\epsilon)$ and an a.s. finite random variable $\Xi$ so that for a.e. $\omega\in\Omega$
         \begin{align}
             \label{est:thm-EM-a.s.}
           \sup_{t\in[0,1]}|Z_t(\omega)-Z_t^n(\omega)|\leq N\Xi(\omega) n^{-(\frac{1}{2}+\frac{\beta}{\alpha(1+\alpha)}\wedge \frac{1}{2})+\epsilon}.
         \end{align}
\end{corollary}
\begin{remark}
 When the noise is non-degenerate, \cite{BDG} shows the strong convergence of first order SDEs \eqref{eq:non-de} and its Euler scheme \eqref{eq:non-de-E} with rate $n^{-(\frac{1}{2}+\frac{\beta}{\alpha}\wedge \frac{1}{2})+\epsilon}$ for $b\in C^\beta(\mR^d)$. For degenerate noise, there are only limited works on Brownian type noise showing either non-quantitative convergence \cite{DF, LM} or slower-rate convergence  \cite{LS18}. As the early try of applying the Euler scheme to singular kinetic SDEs, \cite{HLL} shows that the type of scheme like \eqref{eq:SDE-EM} with taking $\alpha=2$ (i.e. $(L_t)_{t\geq0}$ actually is Brownian motion), therein obtained convergence rate $n^{-\frac{1}{2}+}$ is comparable with non-degenerate case. Apparently
 this also happens here.
\end{remark}

\section{Auxiliary Tools}\label{sec:auxi-tool}
\subsection{Conditional shifted Stochastic Sewing Lemma and weighted John-Nirenberg Inequality} For $S<T$ we denote $[S,T]_{\leq}^i=\left\{(t_0,\ldots,t_{i-1})\in[S,T]^i:\,t_0\leq \cdots\leq t_{i-1}\right\}$, $i\in\mN$,
and $\widehat{[S,T]}_{\leq}^2=\left\{(s,t)\in[S,T]_{\leq}^2:\,|s-t|\leq S\right\}$.
For $0\leq S\leq T$ denote $ \Delta_{[S,T]}:=\{(s,t):S\leq s<t\leq T,s-(t-s)\geq S\}.$
For a function $f: {[S,T]}\rightarrow\mR^{d}$ and for $(s,u,t)\in[S,T]^3_{\leq}$ let $\delta f_{s,u,t}:=f_{s,t}-f_{s,u}-f_{u,t}$.  Denote $\mE^s$ the conditional expectation given $\cF_s$,  $\|X\|_{L_p(\Omega)|\sG}:=(\mE[|X|^p|\sG])^\frac{1}{p}$ for any random variable $X$ and sub-$\sigma$ algebra $\sG\subset\cF$.
\begin{lemma}\cite[Lemma 3.1]{BDG-Levy} \label{lem:SSL}
Let $(S,T)\in[0,1]^2$. Let $p\in[2,\infty)$.   Let
$(A_{s,t})_{(s,t)\in \Delta_{[S,T]}}$ be a family of $\R^d$-valued random variables in $L_p(\Omega)$
such that $A_{s,t}$ is $\mathcal{F}_t$-measurable for all $(s,t)\in \Delta_{[S,T]}$.   Let $\cG\subset\cF_S$ be an $\sigma$-algebra. Suppose that there exist constants $\epsilon_1,\epsilon_2>0$ and constants $K_1, K_2\in[0,\infty)$,  such that the following hold:
\begin{enumerate}[(1)]
   \item $\|A_{s,t}\|_{L_p(\Omega)|\cG}\leq K_1|t-s|^{\frac{1}{2}+\epsilon_1}$, $(s,t)\in\Delta_{[S,T]}$,
   \item $\|\E^{s-(t-s)}\delta A_{s,u,t}\|_{L_p(\Omega)|\cG}\leq K_2|t-s|^{1+\epsilon_2}$, $(s,t)\in\Delta_{[S,T]}$ and $u=\frac{s+t}{2}$,
   \item there exists a  process $\mathcal{A}:[S,T]\mapsto L^p(\Omega;\R^d)$  such that  for any $(s,t)\in [S,T]_\leq^2$ one has
   \begin{align}
       \label{con:SSL3}
       \cA_t-\cA_s=\lim_{m\rightarrow\infty}\sum_{i=1}^{m-1}A_{s+i\frac{t-s}{m},s+(i+1)\frac{t-s}{m}} \quad\text{ in probability}.
   \end{align}
\end{enumerate}
 Then there exists constants $\hat K_1,\hat K_2>0$ which depend only on $\epsilon_1,\epsilon_2,\epsilon_3, d,p$ so that for any $(s,t)\in[S,T]_{\leq}^2$,
  \begin{align}
       \label{con:SSL-est}
       \|\cA_{t}-\cA_{s}\|_{L_p(\Omega)|\cG}\leq K_1\hat K_1|t-s|^{\frac{1}{2}+\epsilon_1}+K_2\hat K_2|t-s|^{1+\epsilon_2}.
   \end{align}
\end{lemma}
\begin{lemma}
    \cite[Proposition 3.2]{BDG-Levy}\label{lem:John}
    Let $0\leq S\leq T$ and let $\cA:\Omega\times[S,T]\mapsto\mR^d$, $\xi:\Omega\times [S,T]\mapsto\mR_+$ be  processes adapted to $(\cF_t)_{t\in[S,T]}$. Further assume  $(\cA_t)_{t\in[S,T]}$ to be continuous and integrable.  Suppose that for $(s,t)\in[S,T]_\leq^2$ one has
    \begin{align}\label{con:John}
    \mE^s|\cA_t-\cA_s|\leq \xi_s\quad a.s.
    \end{align}
    Then for any $p\geq1$ there exists a constant $N=N(p)$ such that for any $(s,t)\in[S,T]_\leq^2$ one has
    \begin{align}\label{est:lem-John}
        \|\sup_{r\in[s,t]}|\cA_t-\cA_s|\|_{L_p(\Omega)|\cF_s}\leq N \|\sup_{r\in[s,t]}\xi_r\|_{L_p(\Omega)|\cF_s}.
    \end{align}

\end{lemma}
\subsection{Estimates on the semi-group}
First, we 
notice that the process $L$ satisfies the following scaling  property: there exists $C(\alpha, p,\epsilon)$ so that
  for any $p\in(0,\alpha)$ and $\epsilon>0$, one has
  \begin{align*}
  \mE[|L_t|^p\wedge 1]\leq Ct^{\frac{p}{\alpha}-\epsilon},\quad t\in(0,1],
  \end{align*}
  which implies that for any $p\geq1, \gamma>0$  there exists $C(p,\epsilon,\gamma,\alpha)$ so that
    \begin{align}
      \label{con:noise-alpha}\big\||L_t|^\gamma\wedge 1\big\|_{L_p(\Omega)}\leq Ct^{\frac{\gamma}{\alpha}\wedge\frac{1}{p}-\epsilon},\quad t\in(0,1].
  \end{align}
 Denote 
\begin{align}M_t:=\left(\begin{smallmatrix}\int_0^t L_r dr\\L_t
    \end{smallmatrix}\right),
    \label{def:noise}
\end{align}
it satisfies
\begin{align}M_t-\Gamma_{t-s}M_s:=\left(\begin{smallmatrix}\int_0^t L_r dr-\int_0^s L_r dr-(t-s)L_s\\L_t-L_s
    \end{smallmatrix}\right)=\left(\begin{smallmatrix}\int_s^t L_r -L_sdr\\L_t-L_s
    \end{smallmatrix}\right)\overset{(d)}{=}\left(\begin{smallmatrix}\int_0^{t-s} L_rdr\\L_{t-s}
    \end{smallmatrix}\right)=M_{t-s}
    \label{prop:Markov}
\end{align}
and $M_t-\Gamma_{t-s}M_s$ is independent of $\cF_s$. 
Notice that
\begin{align*}
  P_tf(z):=\mE f(x+tv+\int_0^tL_r\dif r,v+L_t)=\mE f(M_t(\Gamma_tz)),\qquad t\geq0.
\end{align*}
Moreover \eqref{prop:Markov} implies that  for any bounded measurable function $f:\mR^{2d}\rightarrow \mR$, $(s,t)\in[0,1]_\leq^2$ and any $\cF_s$-measurable  $\mR^{2d}$-valued random variable $\xi$
\begin{align}
    \label{eq:inde-incre} \mE^sf(M_t+\Gamma_t\xi)=P_{t-s}f(M_s+\Gamma_{s}\xi),\quad i.e. \quad  \mE^sf(M_t+\xi)=P_{t-s}f(M_s+\Gamma_{s-t}\xi).
\end{align}
 Let $(\rho_t)_{t\geq0}$ denote the density of $(M_t)_{t\geq0}$ (see more details in \cite[Section 2.2]{HRZ}). One has
 \begin{align}\label{CC01}
P_tf(z)=\Gamma_t \rho_t*\Gamma_tf(z)=\Gamma_t(\rho_t*f)(z),\quad z=(x,v)\in\mR^{2d}.
\end{align}
 One also sees  the  scaling  \cite[(2.21)]{HRZ}
\begin{align}\label{828:00}
\Gamma_t\rho_t(x,v)=t^{-\frac{2d}{\alpha}-d}\rho_1(t^{-\frac{1}{\alpha}-1}x+t^{-\frac{1}{\alpha}}v,t^{-\frac{1}{\alpha}}v)=t^{-\frac{2d}{\alpha}-d}\Gamma_1\rho_1(t^{-\frac{1}{\alpha}-1}x,t^{-\frac{1}{\alpha}}v).
\end{align}

For the operator $\Gamma_t, t\geq0$, we get from \cite[Lemma 3.7]{HLL} \footnote{ $\alpha=2$ in \cite[Lemma 3.7]{HLL}. Going along the same lines of the proof of \cite[Lemma 3.7]{HLL} but replacing index $3$   by $\alpha+1$ generalizes \cite[Lemma 3.7]{HLL} to \eqref{0222:00}.} that
 for $\beta>0$ and $T>0$, for any $\bbp\in[1,\infty]^2$ and $q\in[1,\infty]$, there is a constant $C=C(d,\bbp,T,q)>0$ such that for all $t\in[0,T]$ and $f\in\bB^{(1+\alpha)\beta,q}_{\bbp;\bba}$
\begin{align}\label{0222:00}
    \|\Gamma_tf\|_{\bB^{\beta,q}_{\bbp;\bba}}\le C(\|f\|_{\bB^{\beta,q}_{\bbp;\bba}}+t^\beta\|f\|_{\bB^{(1+\alpha)\beta,q}_{\bbp;\bba}}).
\end{align}
      \begin{lemma}
       For any  $1\geq {\beta_x}, {\beta_v}\ge0$ so that $\frac{\beta_x}{1+\alpha},\beta_v<\alpha-1$ and $m,n\in\mN$ so that $m+n\leq1$, 
       there exist constants $C=C(d,{\beta_x},{\beta_v},m,n)$ such that for any $t\in(0,1]$,
\begin{align}
\label{CC03+}
&\big\||x|^{\frac{\beta_x}{1+\alpha}}|v|^{\beta_v}|\nabla_x^m\nabla_v^n\partial_t(\Gamma_t\rho_t)(x,v)|\big\|_{1}\le C  t^{\frac{\beta_x-(\alpha+1)m+(\beta_v-n)-\alpha}{\alpha}
}.
\end{align}
      \end{lemma}
\begin{proof}
It follows from the scaling \eqref{828:00} that
    \begin{align*}
        \partial_t(\Gamma_t\rho_t)(x,v)=&(-\frac{2d}{\alpha}-d)t^{-\frac{2d}{\alpha}-d-1}\Gamma_1\rho_1(t^{-\frac{1}{\alpha}-1}x,t^{-\frac{1}{\alpha}}v)\\
        &+(-\frac{1}{\alpha}-1)t^{-\frac{1}{\alpha}-2}t^{-\frac{2d}{\alpha}-d}x\cdot \nabla_x \Gamma_1\rho_1(t^{-\frac{1}{\alpha}-1}x,t^{-\frac{1}{\alpha}}v)\\
        &+(-\frac{1}{\alpha})t^{-\frac{1}{\alpha}-1}t^{-\frac{2d}{\alpha}-d}v\cdot \nabla_v \Gamma_1\rho_1(t^{-\frac{1}{\alpha}-1}x,t^{-\frac{1}{\alpha}}v).
    \end{align*}
Let
   \begin{align*}
       G(x,v):=(-\frac{2d}{\alpha}-d)\Gamma_1\rho_1(x,v)+(-\frac{1}{\alpha}-1)x\cdot \nabla_x \Gamma_1\rho_1(x,v)+(-\frac{1}{\alpha})v\cdot \nabla_v \Gamma_1\rho_1(x,v),
   \end{align*}
   one sees that
   \begin{align*}
       \partial_t(\Gamma_t\rho_t)(x,v)=t^{-\frac{2d}{\alpha}-d-1}G(t^{-\frac{1}{\alpha}-1}x,t^{-\frac{1}{\alpha}}v),
   \end{align*}
   which yields
   \begin{align*}
&|x|^\frac{\beta_x}{1+\alpha}|v|^{\beta_v}\nabla_x^m\nabla_v^n\partial_t(\Gamma_t\rho_t)(x,v)\\&=t^{-\frac{2d}{\alpha}-d-1+[\beta_x-(\alpha+1)m+(\beta_v-n)]/\alpha}|t^{-\frac{1}{\alpha}-1}x|^\frac{\beta_x}{1+\alpha}|t^{-\frac{1}{\alpha}}v|^{\beta_v} \nabla_x^m\nabla_v^nG(t^{-\frac{1}{\alpha}-1}x,t^{-\frac{1}{\alpha}}v).
   \end{align*}
   Therefore, we have
   \begin{align*}
&\big\||x|^\frac{\beta_x}{1+\alpha}|v|^{\beta_v}|\nabla_x^m\nabla_v^n\partial_t(\Gamma_t\rho_t)(x,v)|\big\|_{1}\\
&=t^{-\frac{2d}{\alpha}-d-1+[\beta_x-(\alpha+1)m+(\beta_v-n)]/\alpha} t^{\frac{\bba}{\alpha}\cdot \frac{d}{\bbp}}\big\||x|^\frac{\beta_x}{1+\alpha}|v|^{\beta_v}|\nabla_x^m\nabla_v^nG(x,v)|\big\|_{1}\\&\lesssim t^{\frac{\beta_x-(\alpha+1)m+(\beta_v-n)-\alpha}{\alpha}},
   \end{align*}
   provided that $\Gamma_1\rho_1$ is a smooth function and \cite[(5.5)]{HWZ20} (see also \cite[Theorem 1.1 (ii)]{HZ}). This completes the proof.
\end{proof}
\begin{lemma}
    \label{lem:est-semi-sob}
Let $\beta\in(0,\alpha-1)$.    For $f\in\bC_{\bba}^{(1+\alpha)\beta}$
we have for any $0\leq r\leq t\leq 1$,
     \begin{align}\label{est-semi-sob}
     \| \p_r\Gamma_r\rho_r*(\Gamma_tf)\|_{\infty}\lesssim &{\left(\|f\|_{\bC_{\bba}^{\beta}}+t^\beta\|f\|_{\bC_{\bba}^{(1+\alpha)\beta}} \right)}r^{\frac{\beta}{\alpha}-1};
     \\\label{est-semi-sob-1} \| \nabla\p_r\Gamma_r\rho_r*(\Gamma_tf)\|_{\infty}\lesssim &{\left(\|f\|_{\bC_{\bba}^{\beta}}+t^\beta\|f\|_{\bC_{\bba}^{(1+\alpha)\beta}} \right)}r^{\frac{\beta-1}{\alpha}-1}.
 \end{align}
 \end{lemma}

 \begin{proof} Denote $F_t:=\Gamma_tf$.
Observe that
 $\int_{\mR^{2d}}\p_r\Gamma_r\rho_r(z')\dif z'=0 $ due to the fact that $(\rho_r)_{r\geq0}$ is a probability density. By equivalent form \eqref{CH1}
     it reads
     \begin{align*}
         \| \p_r\Gamma_r\rho_r*(\Gamma_tf)\|_{\infty}
  &=
  \|\int_{\mR^{2d}} \p_r\Gamma_r\rho_r(z')\Gamma_tf(z-z')\dif z'\|_{\infty}\\
      &=\|\int_{\mR^{2d}} \p_r\Gamma_r\rho_r(z')(\Gamma_tf(z-z')-\Gamma_tf(z))\dif z'\|_{\infty}
   \\&\lesssim  \|F_t\|_{\bC_{\bba}^{\beta}} \int_{\mR^{2d}}|z'|_\bba^\beta|\p_r\Gamma_r\rho_r(z')|\dif z'
    \\&\lesssim  \|F_t\|_{\bC_{\bba}^{\beta}} \int_{\mR^{2d}}(|x'|^\frac{\beta}{1+\alpha}+|v'|^\beta)|\p_r\Gamma_r\rho_r(z')|\dif z'
 \\&\lesssim  \|F_t\|_{\bC_{\bba}^{\beta}} r^{(\frac{\beta}{\alpha}-1)}.
     \end{align*}
    The first and second inequalities are following from elementary calculation.
    We applied  \eqref{CC03+}  in the last  inequality twice with taking $\beta_x=\frac{\beta}{1+\alpha}, \beta_v=0$, $m=n=0$, firstly and  $\beta_x=0, \beta_v=\beta$, $m=n=0$ in the second time.  We finally get \eqref{est-semi-sob} by \eqref{0222:00}.
     For \eqref{est-semi-sob-1}, we get similarly via applying \eqref{CC03+} since 
     \begin{align*}
               \| \nabla\p_r\Gamma_r\rho_r*F_t\|_{\bbp} \lesssim&
         \|F_t\|_{\bC_{\bba}^{\beta}} \int_{\mR^{2d}}|z'|_\bba^{\beta}|\nabla\p_r\Gamma_r\rho_r(z')|\dif z'     
         \\\lesssim & \|F_t\|_{\bC_{\bba}^{\beta}} \int_{\mR^{2d}}(|x'|^{\frac{\beta}{1+\alpha}}+|v'|^\beta)(|\nabla_x\p_r\Gamma_r\rho_r(z')|+|\nabla_v\p_r\Gamma_r\rho_r(z')|)\dif z'
    \\\lesssim &   \|F_t\|_{\bC_{\bba}^{\beta}}r^{\frac{\beta-1}{\alpha}-1}.
     \end{align*}
     The proof completes.
     \end{proof}
\begin{proposition}\label{prop:semi}
For $\epsilon>0, \beta\in(0,\alpha-1),$ some constant $\gamma$ as below, $\epsilon<\mu\leq 1+\epsilon$ and $\mu\geq \frac{\beta-\gamma}{\alpha}$, there exists a constant $N=N(d,\alpha,\beta,\mu,\gamma,\epsilon)$ so that the following holds: for any $f\in \bC_\bba^{(1+\alpha)\beta}$ and $(s,t)\in(0,1]_\leq^2$,
\begin{align}
    \label{est:semi-group-1}
    \|\nabla^\gamma P_tf\|_{\infty}\leq& N \|f\|_{\bC_\bba^{\beta}}t^{\frac{(\beta-\gamma)\wedge0}{\alpha}}, \quad t\in(0,1], \gamma\geq0;\\
     \label{est:semi-group-2}
    \|\nabla^\gamma (P_tf-P_s\Gamma_{t-s}f)\|_{\infty}\leq& N (\|f\|_{\bC_\bba^{\beta}}+t^{\beta}\|f\|_{\bC_\bba^{(1+\alpha)\beta}} )s^{\frac{(\beta-\gamma)}{\alpha}-\mu}(t-s)^{\mu-\epsilon}, \gamma=0,1.
\end{align}
\end{proposition}
\begin{proof}
    \eqref{est:semi-group-1} is directly from \cite[Lemma 2.16, Remark 2.17]{HRZ}. Here we only need to show  \eqref{est:semi-group-2}. 
    By \eqref{CC01}, one sees that
 \begin{align*}
P_tf-P_s\Gamma_{t-s}f=(\Gamma_t\rho_t)*(\Gamma_tf)-(\Gamma_s\rho_s)*(\Gamma_tf)=\int_s^t\p_r\Gamma_r\rho_r\dif r\ast F_t,   \quad F_t:=\Gamma_tf.
 \end{align*}
 Then from Lemma \ref{lem:est-semi-sob} 
 we have
\begin{align*}
   \|\nabla^\gamma (P_tf-P_s\Gamma_{t-s}f)\|_{\infty}=&\|\int_s^t (\nabla^\gamma \p_r\Gamma_r\rho_r)\ast F_t\dif r\|_{\infty}\leq\int_s^t \|(\nabla^\gamma \p_r\Gamma_r\rho_r)\ast F_t\|_{\infty}\dif r
   \\ \lesssim&(\|f\|_{\bC_\bba^{\beta}}+t^{\beta}\|f\|_{\bC_\bba^{(1+\alpha)\beta}} )\int_s^tr^{\frac{\beta-\gamma}{\alpha}-1}dr \\ \lesssim& (\|f\|_{\bC_\bba^{\beta}}+t^{\beta}\|f\|_{\bC_\bba^{(1+\alpha)\beta}} )s^{\frac{(\beta-\gamma)}{\alpha}-\mu}(t-s)^{\mu-\epsilon}.
 \end{align*}
 Last inequality holds since $\frac{(\beta-\gamma)}{\alpha}-\mu\leq0$, $\mu-1>-1$ and $|t-s|\leq1.$
 We get the desired results.
\end{proof}

\subsection{A priori bounds}
For $p\geq 1$, $(S,T)\in[0,1]_\leq^2$, $\gamma\in(0,1]$, for a $\sigma$-algebra $\sG\subset \mathcal{F}_S$, for a bounded measurable function $f:[0,1]\times\Omega\rightarrow\mR^d$ which is adapted to $\mathcal{F}$, we denote
\begin{align}
    \|f\|_{C_p^0|\sG,[S,T]}:=\sup_{s\neq t\in[S,T]}\big\|f(t)\big\|_{L_p(\Omega)|\sG};\label{def:norm-random-0}\\
    [f]_{C_p^\gamma|\sG,[S,T]}:=\sup_{s\neq t\in[S,T]}\frac{\|f(t)-f(s)\|_{L_p(\Omega)|\sG}}{|t-s|^\gamma};\label{def:norms-random-1}\\
      [f]_{\sC_p^\gamma|\sG,[S,T]}:=\sup_{s\neq t\in[S,T]}\frac{\big\|\|f(t)-\mE^sf(t)\|_{L_1(\Omega)|\cF_s}\big\|_{L_p(\Omega)|\sG}}{|t-s|^\gamma}.\label{def:norms-random-2}
\end{align}
According to \cite[Lemma 4.2]{BDG-Levy}
we know that for any random variable $\tilde Y\in L_p(\Omega)$ and any $\sG$-measurable random variable $\hat{Y}\in L_p(\Omega)$
\begin{align}
    \|\tilde Y-\mE[\tilde Y|\sG]\|_{L_p(\Omega)|\sG}\leq2\|\tilde Y-\hat{Y}\|_{L_p(\Omega)|\sG} ,\quad a.s.  \label{est:two-random-norms-1}
\end{align}
and for (semi-)norms from \eqref{def:norms-random-1} and \eqref{def:norms-random-2} one has
\begin{align}
    \label{est:two-random-norms-2} [f]_{\sC_p^\gamma|\sG,[S,T]}\leq 2  [f]_{C_p^\gamma|\sG,[S,T]}.
\end{align}

\begin{lemma}
    \label{lem:moment-M}
    Let $M$ be as \eqref{def:noise}. Then for any $p\geq1,\gamma>0$, for any $(s,t)\in[0,1]_\leq^2$,  there  exists $C=C(\gamma,p,d,\alpha,\epsilon)>0$ so that
    \begin{align}
        \label{est:moment-M}
         \big\||M_t-\Gamma_{t-s}M_s|_\bba^\gamma\wedge1\big\|_{L_p(\Omega)}  \leq C|t-s|^{\frac{\gamma}{\alpha}\wedge\frac{1}{p}-\epsilon}.
    \end{align}
\end{lemma}
\begin{proof}
  By \eqref{prop:Markov} and Minkowski inequality we can write  (without losing generality we can assuming  $p\frac{\gamma}{1+\alpha}>1$)
       \begin{align*}
         \big\|&|M_t-\Gamma_{t-s}M_s|_\bba^\gamma\wedge1\big\|_{L_p(\Omega)}\\ = &  \big\||M_{t-s}|_\bba^\gamma\wedge1\big\|_{L_p(\Omega)}\lesssim \big\|(|\int_0^{t-s}L_rdr|^{\frac{\gamma}{\alpha+1}}+|L_{t-s}|^\gamma)\wedge1\big\|_{L_p(\Omega)}
         \\\lesssim& \big(\int_0^{t-s}\big\|L_r\big\|_{L_{\frac{p\gamma}{1+\alpha}}(\Omega)}\wedge 1dr\big)^{\frac{\gamma}{\alpha+1}}+\big\||L_{t-s}|^\gamma\wedge 1\big\|_{L_p(\Omega)}
          \\\lesssim&|t-s|^{\frac{\gamma}{\alpha+1}(1+\frac{1}{\alpha})\wedge\frac{1}{p}-\epsilon}+|t-s|^{\frac{\gamma}{\alpha}\wedge\frac{1}{p}-\epsilon}
           \lesssim  |t-s|^{\frac{\gamma}{\alpha}\wedge\frac{1}{p}-\epsilon}.
    \end{align*}
    The penultimate inequality follows from \eqref{con:noise-alpha}.
\end{proof}
\begin{lemma}\label{lem:apriori}
    Let $p\geq 1$ and $\epsilon>0$. Suppose that  Assumption \ref{ass:main1} holds. Let $M$ be as \eqref{def:noise} and let $Z=(X,V)$ and $Z^n=(X^n,V^n)$ satisfy \eqref{eq:SDE} and \eqref{eq:SDE-EM} correspondingly. Then for $Y:=Z-M, Y^n:=Z^n-M$ there exists a constant $C=C(d,\beta,\|b\|_{\bC_\bba^{\beta}},\alpha,\epsilon,p)$ so that
    \begin{align}
   \|Y_t-\mE^sY_t  \|_{L_p(\Omega)|\cF_s}  \leq C|t-s|^{1+(\frac{\beta}{\alpha}\wedge\frac{1}{p})-\epsilon}\label{est:apaiori-1},\\
         \|Y_t^n-\mE^sY_t^n  \|_{L_p(\Omega)|\cF_s}  \leq C|t-s|^{1+(\frac{\beta}{\alpha}\wedge\frac{1}{p})-\epsilon}.  \label{est:apaiori-2}
    \end{align}
\end{lemma}
\begin{proof}
We apply the induction following from \cite[Proof of Lemma 4.2 (iii)]{BDG-Levy}. Suppose \eqref{est:apaiori-1} holds for some $m\geq0$:
\begin{align*}
     \|Y_t-\mE^sY_t  \|_{L_p(\Omega)|\cF_s}  \leq C|t-s|^m.
\end{align*}
This is clearly true since $b$ is bounded and continuous.  We proceed now inductively on $m$.
We apply \eqref{est:two-random-norms-1} with taking $\sG:=\cF_s$ and $\tilde Y:=Y_t, \hat{Y}:=\Gamma_{t-s}Y_s+\int_s^tB(\Gamma_{r-s}M_s+\mE^sY_r)dr\in\sG$ where $B(x,v):=\left(\begin{smallmatrix} v\\b(x,v)
    \end{smallmatrix}\right)$, $(x,v)\in\mR^{d}\times\mR^{d}$. By the fact that $Z=(X,V)$ satisfies \eqref{eq:SDE}
    we get for $(s,t)\in[0,1]_\leq^2$,
    \begin{align*}
   \|&Y_t-\mE^sY_t  \|_{L_p(\Omega)|\sF_s}\\\leq  & 2  \big\|Y_t-\Gamma_{t-s}Y_s-\int_s^tB(\Gamma_{r-s}M_s+\mE^sY_r)dr\big\|_{L_p(\Omega)|\sF_s}
   \\=&2 \Big\|\left(\begin{smallmatrix} \int_0^t\int_0^r b(Z_u)dudr-\int_0^s\int_0^r b(Z_u)dudr-(t-s)\int_0^s b(Z_u)du\\\int_0^t b(Z_u)du-\int_0^s b(Z_u)du
    \end{smallmatrix}\right)\\&\quad\qquad\qquad\qquad\qquad\qquad\qquad\qquad-\int_s^tB(\Gamma_{r-s}M_s+\mE^sY_r)dr
   \Big\|_{L_p(\Omega)|\sF_s}
     \\=&2 \Big\|\left(\begin{smallmatrix} \int_s^t\int_s^r b(Z_u)dudr\\\int_s^t b(Z_u)du
    \end{smallmatrix}\right)-\int_s^tB(\Gamma_{r-s}M_s+\mE^sY_r)dr
  \Big\|_{L_p(\Omega)|\sF_s}
    \\ {\lesssim}&{\|b\|_{\bC_\bba^\beta}}
    \big\|\int_s^t\int_s^r\big(|M_u-\Gamma_{u-s}M_s+Y_u-\mE^sY_u|_\bba^\beta\big)\wedge1 dudr\big\|_{L_p(\Omega)|\sF_s} \\& \quad\qquad\qquad+\|b\|_{\bC_\bba^\beta}\big\|\int_s^t\big(|M_r-\Gamma_{r-s}M_s+Y_r-\mE^sY_r|_\bba^\beta\big)\wedge1 dr\big\|_{L_p(\Omega)|\sF_s}
     \\\lesssim& \|b\|_{\bC_\bba^\beta}
   \int_s^t\int_s^t \big\|\big(|M_u-\Gamma_{u-s}M_s|_\bba^\beta+|Y_u-\mE^sY_u|_\bba^\beta\big)\wedge1\big\|_{L_p(\Omega)|\sF_s} dudr \\&\quad\qquad\qquad +\|b\|_{\bC_\bba^\beta}\int_s^t\big\|\big(|M_r-\Gamma_{r-s}M_s|_\bba^\beta+|Y_r-\mE^sY_r|_\bba^\beta\big)\wedge1 \big\|_{L_p(\Omega)|\sF_s}dr\\\lesssim& 
 \|b\|_{\bC_\bba^\beta}\int_s^t\big\|\big(|M_r-\Gamma_{r-s}M_s|_\bba^\beta+|Y_r-\mE^sY_r|_\bba^\beta\big)\wedge1 \big\|_{L_p(\Omega)|\sF_s}dr .
\end{align*}
Following from \eqref{prop:Markov}, Lemma \ref{lem:moment-M} and  the induction hypothesis we have
\begin{align*}
    \big \||Y_t-\mE^sY_t|_\bba\big\|_{L_q(\Omega)|\cF_s}  &\leq C|t-s|^{2+(\frac{\beta}{(1+\alpha)\alpha}\wedge\frac{1}{p}\wedge\beta m)-\epsilon}+ C|t-s|^{1+(\frac{\beta}{\alpha}\wedge\frac{1}{p}\wedge\beta m)-\epsilon}\\&\leq C|t-s|^{1+(\frac{\beta}{\alpha}\wedge\frac{1}{p}\wedge\beta m)-\epsilon}.
\end{align*}
Observe that if $\epsilon>0$ is sufficiently small then the recursion $m_0=0,m_{j+1}=1+(\frac{\beta}{\alpha}\wedge\frac{1}{p}\wedge\beta m_j)-\epsilon, j\in\mN$ reaches $1+(\frac{\beta}{\alpha}\wedge\frac{1}{p})-\epsilon$ in finitely many steps since $\alpha>1-\beta,\beta<1$ implying $\alpha>\frac{1-\beta}{1-\epsilon}$ for small enough $\epsilon$. Hence \eqref{est:apaiori-1} holds.\\

We show \eqref{est:apaiori-2} similarly by induction. This time we apply \eqref{est:two-random-norms-1} by taking $\sG:=\cF_s$ and $\tilde Y:=Y_t^n, \hat{Y}:=\Gamma_{t-\hat{k}_{n}(s)}Y_{\hat{k}_{n}(s)}^n+\int_{\hat{k}_{n}(s)}^tB(\Gamma_{{k_n(r)}-s}M_s+\mE^sY_{k_n(r)}^n)dr$ where  $\hat{k}_{n}(s):=\frac{\lceil ns\rceil}{n}$. To be noticed that by definition of $Y^n$, $Y^n_{\hat{k}_{n}(s)}\in\sG$ although $\hat{k}_{n}(s)\geq s$.  Then we can write
for $(s,t)\in[0,1]_\leq^2$
    \begin{align*}
  \big \|&Y_t^n-\mE^s[Y_t^n]\big\|_{L_p(\Omega)|\sF_s}\\\leq  & 2  \big\|Y_t^n-\Gamma_{t-\hat{k}_{n}(s)}Y_{\hat{k}_{n}(s)}^n-\int_{\hat{k}_{n}(s)}^tB(\Gamma_{{k_n(r)}-s}M_s+\mE^sY_{k_n(r)}^n)dr\big\|_{L_p(\Omega)|\sF_s}
   \\=&2 \Big\|\left(\begin{smallmatrix} \int_0^t\int_0^r b(Z_{{k}_{n}(u)}^n)dudr-\int_0^{\hat{k}_{n}(s)}\int_0^rb(Z_{{k}_{n}(u)}^n)dudr-(t-\hat{k}_{n}(s))\int_0^{\hat{k}_{n}(s)} b(Z_{{k}_{n}(u)}^n)du\\\int_0^t b(Z_{{k}_{n}(u)}^n)du-\int_0^{\hat{k}_{n}(s)} b(Z_{{k}_{n}(u)}^n)du
    \end{smallmatrix}\right)\\&\quad-\int_{\hat{k}_{n}(s)}^tB(\Gamma_{{k_n(r)}-s}M_s+\mE^sY_{k_n(r)}^n)dr
\Big\|_{L_p(\Omega)|\sF_s}
     \\=&2 \Big\|\left(\begin{smallmatrix} \int_{\hat{k}_{n}(s)}^t\int_{\hat{k}_{n}(s)}^r b(Z_{{k}_{n}(u)}^n)dudr\\\int_{\hat{k}_{n}(s)}^t b(Z_{{k}_{n}(u)}^n)du
    \end{smallmatrix}\right)-\int_{\hat{k}_{n}(s)}^tB(\Gamma_{{k_n(r)}-s}M_s+\mE^sY_{{k_n(r)}}^n)dr
  \Big\|_{L_p(\Omega)|\sF_s}
    \\\lesssim&
    \big\|\int_{\hat{k}_{n}(s)}^t\int_{\hat{k}_{n}(s)}^r\big(|M_u-\Gamma_{{k_n(u)}-s}M_s|_\bba^\beta+|Y_{{k_n(u)}}^n-\mE^sY_{{k_n(u)}}^n|_\bba^\beta\big)\wedge1 dudr\big\|_{L_p(\Omega)|\sF_s}\\& +\big\|\int_{\hat{k}_{n}(s)}^t\big(|M_{k_n(r)}-\Gamma_{{k_n(r)}-s}M_s|_\bba^\beta+|Y_{{k_n(r)}}^n-\mE^sY_{{k_n(r)}}^n|_\bba^\beta\big)\wedge1 dr\big\|_{L_p(\Omega)|\sF_s}
     \\\lesssim&
  \int_{\hat{k}_{n}(s)}^t\big\|\big(|M_r-\Gamma_{{k_n(r)}-s}M_s|_\bba^\beta+|Y_{{k_n(r)}}^n-\mE^sY_{{k_n(r)}}^n|_\bba^\beta\big)\wedge1 \big\|_{L_p(\Omega)|\sF_s}dr.
\end{align*}
We can see that for $r\geq \hat{k}_{n}(s)$, ${k}_{n}(r)\geq \hat{k}_{n}(s)\geq s$ so $M_r-\Gamma_{{k_n(r)}-s}M_s$ is independent of $\cF_s$. By the same argument of getting \eqref{est:apaiori-1} we get \eqref{est:apaiori-2}.
\end{proof}
\begin{lemma}
    \label{lem:est-on-Yn} Let $p\geq 1$ and $\epsilon>0$. Suppose that Assumption \ref{ass:main1} holds. Let $M$ be as \eqref{def:noise} and let  $Z^n=(X^n,V^n)$ satisfy \eqref{eq:SDE-EM}. Then for $ Y^n:=Z^n-M$ there exists a constant $N=N(d,\beta,\|b\|_{\bC_\bba^\beta},\alpha,\epsilon,p)$ so that for any $(s,t)\in[0,1]_\leq^2$
    \begin{align}
        \label{est:Yn-Yknt-1}
       &  [Y^n-\Gamma_{\cdot-k_n(\cdot)}Y^n_{k_n(\cdot)}]_{C_p^0|\sG,[s,t]}\leq Nn^{-1}; \\&
      [Y^n-\Gamma_{\cdot-k_n(\cdot)}Y^n_{k_n(\cdot)}]_{\sC_p^{\frac{1}{2}}|\sG,[s,t]}\leq N n^{-\frac{1}{2}-(\frac{\beta}{\alpha}\wedge\frac{1}{2})+\epsilon} .\label{est:Yn-Yknt-2}
    \end{align}
\end{lemma}
\begin{proof}
    We start with showing \eqref{est:Yn-Yknt-1}. First notice that  for any $t\in[0,1]$
    \begin{align}\label{equ:Yn}
      Y_t^n-\Gamma_{t-k_n(t)}Y^n_{k_n(t)} = &\left(\begin{smallmatrix} \int_0^t\int_0^r b(Z_{{k}_{n}(u)}^n)dudr-\int_0^{{k}_{n}(t)}\int_0^rb(Z_{{k}_{n}(u)}^n)dudr-(t-{k}_{n}(t))\int_0^{{k}_{n}(t)} b(Z_{{k}_{n}(u)}^n)du \\\int_0^t b(Z_{{k}_{n}(u)}^n)du-\int_0^{{k}_{n}(t)} b(Z_{{k}_{n}(u)}^n)du
    \end{smallmatrix}\right)
    \nonumber\\=&\left(\begin{smallmatrix} \int_{{k}_{n}(t)}^t\int_{{k}_{n}(t)}^r b(Z_{{k}_{n}(u)}^n)dudr \\\int_{{k}_{n}(t)}^t b(Z_{{k}_{n}(u)}^n)du
    \end{smallmatrix}\right)=\left(\begin{smallmatrix} \frac{(t-k_n(t))^2}{2}b(Z_{{k}_{n}(t)}^n) \\(t-k_n(t)) b(Z_{{k}_{n}(t)}^n)
    \end{smallmatrix}\right).
    \end{align}
    Since $b$ is bounded, we easily  get \eqref{est:Yn-Yknt-1} by the definition \eqref{def:norm-random-0}.

    We move to show \eqref{est:Yn-Yknt-2}. Take $(r_1,r_2)\in[s,t]_\leq^2$. If $r_1\geq k_n(r_2)$, then $Y_{r_2}^n$ and $Y^n_{k_n(r_2)}$ are both $\cF_{r_1}$-measurable and it yields
    \begin{align*}
        (Y_{r_2}^n-\Gamma_{r_2-k_n(r_2)}Y^n_{k_n(r_2)})-\mE^{r_1}(Y_{r_2}^n-\Gamma_{r_2-k_n(r_2)}Y^n_{k_n(r_2)})=0,
    \end{align*}
    hence
    \begin{align}
        \label{est:Yn-case1}\mE^{r_1}|  (Y_{r_2}^n-\Gamma_{r_2-k_n(r_2)}Y^n_{k_n(r_2)})-\mE^{r_1}(Y_{r_2}^n-\Gamma_{r_2-k_n(r_2)}Y^n_{k_n(r_2)})|=0\quad \text{ if } r_1\geq k_n(r_2).
    \end{align}
    Otherwise, i.e. $s\leq r_1\leq k_n(r_2)\leq r_2\leq t$, we have  from \eqref{equ:Yn} and \eqref{est:two-random-norms-1} with taking $\tilde Y$ to be $b(Z_{{k}_{n}(r_2)}^n)$ here and $\hat Y$ to be $b(Z_{r_2}^n)$   and $\sG$ to be $\cF_{r_1}$ that
    \begin{align*}
 & \mE^{r_1}|  (Y_{r_2}^n-\Gamma_{r_2-k_n(r_2)}Y^n_{k_n(r_2)})-\mE^{r_1}(Y_{r_2}^n-\Gamma_{r_2-k_n(r_2)}Y^n_{k_n(r_2)})|\\&=       \mE^{r_1}\Big|\left(\begin{smallmatrix} \frac{(r_2-k_n(r_2))^2}{2}b(Z_{{k}_{n}(r_2)}^n) \\(r_2-k_n(r_2)) b(Z_{{k}_{n}(r_2)}^n)
    \end{smallmatrix}\right)-\mE^{r_1}\left(\begin{smallmatrix} \frac{(r_2-k_n(r_2))^2}{2}b(Z_{{k}_{n}(r_2)}^n) \\(r_2-k_n(r_2)) b(Z_{{k}_{n}(r_2)}^n)
    \end{smallmatrix}\right)\Big|
    \\&\leq (r_2-k_n(r_2)) \mE^{r_1}\Big|\left(\begin{smallmatrix} \frac{(r_2-k_n(r_2))}{2}\big(b(Z_{{k}_{n}(r_2)}^n)-\mE^{r_1}b(Z_{{k}_{n}(r_2)}^n) \big)\\ b(Z_{{k}_{n}(r_2)}^n)-\mE^{r_1}b(Z_{{k}_{n}(r_2)}^n)
    \end{smallmatrix}\right)\Big|
    \\&\leq (r_2-k_n(r_2)) \mE^{r_1}\Big|\left(\begin{smallmatrix} \frac{(r_2-k_n(r_2))}{2}\big(b(Z_{{k}_{n}(r_2)}^n)-b(Z_{r_2}^n) \big)\\ b(Z_{{k}_{n}(r_2)}^n)-b(Z_{r_2}^n)
    \end{smallmatrix}\right)\Big|
    \\&\leq  (r_2-k_n(r_2))\|b\|_{\bC_\bba^\beta }\mE^{r_1}\big|\big(|M_{r_2}-\Gamma_{{r_2-k_n(r_2)}}M_{k_n(r_2)}|_\bba^{\beta\wedge1}+|Y_{{{r_2}}}^n-\Gamma_{{r_2-k_n(r_2)}}Y_{{k_n(r_2)}}^n|_\bba^{\beta\wedge1}\big)\wedge1\big|
    \\&\leq N (r_2-k_n(r_2))(|r_2-k_n(r_2)|^{\frac{\beta}{\alpha}-\epsilon}+|r_2-k_n(r_2)|^{\beta}),
    \end{align*}
    where the last inequality follows from the fact $Y^n$ is Lipschitz continuous since $b$ is bounded.  Clearly $\frac{\beta}{\alpha}\leq \beta$, therefore
    \begin{align} \label{est:Yn-case2}
       \| \mE^{r_1}|  (Y_{r_2}^n-\Gamma_{r_2-k_n(r_2)}Y^n_{k_n(r_2)})&-\mE^{r_1}(Y_{r_2}^n-\Gamma_{r_2-k_n(r_2)}Y^n_{k_n(r_2)})|\|_{L_p(\Omega)|\sG}\nonumber\\&\leq
    N (r_2-k_n(r_2))|r_2-k_n(r_2)|^{\frac{\beta}{\alpha}-\epsilon}\nonumber\\&\leq Nn^{-\frac{1}{2}-(\frac{\beta}{\alpha}\wedge\frac{1}{2})+\epsilon}|r_2-r_1|^\frac{1}{2}\quad\text{ if } r_1<k_n(r_2).
    \end{align}
    The last inequality holds since $r_2-k_n(r_2)\leq n^{-1}\wedge r_2-r_1$.

    Finally  \eqref{est:Yn-case1} together with \eqref{est:Yn-case2} and \eqref{def:norms-random-2} yield \eqref{est:Yn-Yknt-2}.
\end{proof}

\section{Estimates of the  functional}\label{sec:app-term}
Take $Z, Z^n$ from \eqref{eq:SDE} and \eqref{eq:SDE-EM} correspondingly. Notice that for $Y:=Z-M$ and $Y^n:=Z^n-M$  where $M$ is from \eqref{def:noise},
it writes
\begin{align}
    \label{eq:main-error}&(Z_t-Z_t^n)-(Z_s-Z_s^n)=\left(\begin{smallmatrix}(X_t-X_t^n)-(X_s-X_s^n)\\(V_t-V_t^n)-(V_s-V_s^n)
    \end{smallmatrix}\right)=\left(\begin{smallmatrix}\int_s^tV_r-V_r^n dr\\\int_s^tb(Z_r)-\Gamma_{r-k_n(r)}b(Z_{k_n(r)}^n)dr
    \end{smallmatrix}\right)
\nonumber    \\=&\left(\begin{smallmatrix}\frac{1}{3}\int_s^tV_r-V_r^n dr\\\int_s^tb(Y_r+M_r)-b(Y^n_r+M_r) dr
    \end{smallmatrix}\right)
  +\left(\begin{smallmatrix}\frac{1}{3}\int_s^tV_r-V_r^n dr\nonumber \\\int_s^tb(Y^n_{r}+M_r)-b(\Gamma_{r-k_n(r)}Y^n_{k_n(r)}+ M_r)dr \end{smallmatrix}\right)
  \nonumber \nonumber    \\& +\left(\begin{smallmatrix}\frac{1}{3}\int_s^tV_r-V_r^n dr\\\int_s^tb(\Gamma_{r-k_n(r)}Y^n_{k_n(r)}+ M_r)-\Gamma_{r-k_n(r)}b(Y^n_{k_n(r)}+M_{k_n(r)})dr
    \end{smallmatrix}\right)
  \nonumber
    \\=&:\left(\begin{smallmatrix}\frac{1}{3}\int_s^tV_r-V_r^n dr\\S_{s,t}^{b,n,1}
    \end{smallmatrix}\right)+\left(\begin{smallmatrix}\frac{1}{3}\int_s^tV_r-V_r^n dr\\S_{s,t}^{b,n,2}
    \end{smallmatrix}\right)+\left(\begin{smallmatrix}\frac{1}{3}\int_s^tV_r-V_r^n dr\\S_{s,t}^{b,n,3}
    \end{smallmatrix}\right).
\end{align}
\subsection{Estimates of $S^{b,n,1}$, $S^{b,n,2}$}
\begin{lemma}
    \label{lem:auxli-S1S2} Let $p\geq2$.  Let $\gamma_1,\gamma_2,\gamma_3\in(0,1]$ be constants so that
    \begin{align}
        \label{con:constants} 
        \frac{\beta-2}{\alpha}+\gamma_1>0,\quad \gamma_2+\frac{\beta-1}{\alpha}>0.
    \end{align}
    Let $f\in\mC_\bba ^{(1+\alpha)\beta}$, $\beta\in(0,\alpha-1)$. Let $\phi,\psi:[0,1]\time\Omega\rightarrow\mR^{2d}$ be bounded, adapted, measurable functions and suppose that there exists $C_\phi, C_\psi>0$ so that for any $(s,t)\in[0,1]_\leq^2$ we have a.s. 
    \begin{align}
       \mE^s|\phi_t-\mE^s\phi_t|\leq C_\phi|t-s|^{\gamma_1}, \label{con:S1S2-1}
    \\ \mE^s|\psi_t-\mE^s\psi_t|  \leq C_\psi|t-s|^{\gamma_3}.\label{con:S1S2-2}
    \end{align}
    Then there exists a constant $N =N(\alpha,\gamma,p,d,\gamma_1,\gamma_2, \beta)$ so that for any $(S,T)\in[0,1]_\leq^2$ and any $\sigma$-algebra $\sG\subset\cF_S$ we have the bound
    \begin{align}
        \label{est:S1S2}
      \| \int_S^T f(M_r+\phi_r)-&f(M_r+\psi_r)dr\|_{L_p(\Omega)|\sG}\nonumber\\\leq &N\|f\|_{\mC_\bba^{\beta}} |T-S|^{\frac{(\beta-1)}{\alpha}+\gamma_2+1}[\phi-\psi]_{\sC_p^{\gamma_2}|\sG,[S,T]}\nonumber\\&+N\|f\|_{\mC_\bba^{\beta}} |T-S|^{\frac{(\beta-1)}{\alpha}+1}\|\phi-\psi\|_{C_p^0|\sG,[S,T]}\nonumber\\&+NC_\phi\|f\|_{\mC_\bba^{\beta}} |T-S|^{\frac{(\beta-2)}{\alpha}+1+\gamma_1}\|\phi-\psi\|_{C_p^0|\sG,[S,T]}.
    \end{align}
\end{lemma}
\begin{proof}
The idea is to show \eqref{est:S1S2} via Lemma \ref{lem:SSL} (\cite{BDG-Levy}).    Denote for any $(s,t)\in\Delta_{[S,T]}$,
    \begin{align*}
        A_{s,t}:=\mE^{s-(t-s)}\int_s^tf(M_r+\mE^{s-(t-s)}\phi_r)-f(M_r+\mE^{s-(t-s)}\psi_r)dr,\\
        \cA_{t}:=\int_0^tf(M_r+\phi_r)-f(M_r+\psi_r)dr.
    \end{align*}
    In order to apply Lemma \ref{lem:SSL} we verify its conditions. Observe that from \eqref{eq:inde-incre}
    \begin{align*}
        | A_{s,t}|=&\Big|\int_s^tP_{r-(s-(t-s))}f(M_{s-(t-s)}+\Gamma_{s-(t-s)-r}\mE^{s-(t-s)}\phi_r)\\&\qquad-P_{r-(s-(t-s))}f(M_{s-(t-s)}+\Gamma_{s-(t-s)-r}\mE^{s-(t-s)}\psi_r)dr\Big|
        \\\leq& \int_s^t\|P_{r-(s-(t-s))}f\|_{\bC^1}|\mE^{s-(t-s)}(\phi_r-\psi_r)| dr
        .
    \end{align*}
  \eqref{est:semi-group-1} taking $\gamma=1$ together with the property of conditional expectation  which implies  $\|\mE^{s-(t-s)}(\phi_r-\psi_r)\|_{L_p(\Omega)|\sG}\leq \||\phi_r-\psi_r| \|_{L_p(\Omega)|\sG}$ for $\sG\subset\cF_S\subset\cF_{s-(t-s)}$ gives us
    \begin{align*}
       & \|A_{s,t}\|_{L_p(\Omega)|\sG}\\\lesssim& \|f\|_{\bC_\bba^{\beta}}\int_s^t (r-s)^{\frac{(\beta-1)}{\alpha}}\|\mE^{s-(t-s)}(\phi_r-\psi_r) \|_{L_p(\Omega)|\sG}dr
        \\\lesssim& \|f\|_{\bC_\bba^{\beta}}(t-s)^{1+\frac{(\beta-1)}{\alpha}}\sup_{r\in[S,T]}\|\mE^{s-(t-s)}(\phi_r-\psi_r)\|_{L_p(\Omega)|\sG}
         \\\lesssim& \|f\|_{\bC_\bba^{\beta}}(t-s)^{1+\frac{(\beta-1)}{\alpha}}\sup_{r\in[S,T]}\|\phi_r-\psi_r \|_{L_p(\Omega)|\sG}
         \\\leq&N\|f\|_{\bC_\bba^{\beta}}(t-s)^{1+\frac{(\beta-1)}{\alpha}}\|\phi-\psi\|_{C_p^0|\sG,[S,T]}.
    \end{align*}
    Since $\beta>\frac{2-\alpha}{2}$ we have $1+\frac{(\beta-1)}{\alpha}>\frac{1}{2}$. Hence Lemma \ref{lem:SSL} (1) holds with $K_1=N\|f\|_{\bC_\bba^{\beta}}\|\phi-\psi\|_{C_p^0|\sG,[S,T]}.$

    Now let us verify Lemma \ref{lem:SSL} (2).  Let $(s,t)\in\Delta_{[S,T]}$ and $u:=\frac{t+s}{2}$, for convenience also denote $s_1:=s-(t-s)$, $s_2:=s-(u-s)$, $s_3:=s$,  $s_4:=u$, $s_5:=t$ so that one has $s_1\leq s_2\leq s_3\leq  s_4\leq s_5$.
    We can write, notice $u-(t-s)=s=s_3$,
    \begin{align*}
        \mE^{s-(t-s)}\delta A_{s,u,t}=&\mE^{s_1}\delta A_{s_3,s_4,s_5}
        \\=& \mE^{s_1}\int_{s_3}^{s_4} \big(f(M_r+\mE^{s_1}\phi_r)-f(M_r+\mE^{s_1}\psi_r) \\&\qquad-(f(M_r+\mE^{s_2}\phi_r)-f(M_r+\mE^{s_2}\psi_r))\big)dr
        \\&+\mE^{s_1}\int_{s_4}^{s_5}\big(f(M_r+\mE^{s_1}\phi_r)-f(M_r+\mE^{s_1}\psi_r) \\&\qquad-(f(M_r+\mE^{s_3}\phi_r)-f(M_r+\mE^{s_3}\psi_r))\big)dr=:S_1+S_2.
    \end{align*}
    For $S_1$, we have from \eqref{prop:Markov} and \eqref{eq:inde-incre} that
    \begin{align*}
        S_1=& \mE^{s_2}\mE^{s_1}\int_{s_3}^{s_4} \big(f(M_r+\mE^{s_1}\phi_r)-f(M_r+\mE^{s_1}\psi_r) \\&\qquad-(f(M_r+\mE^{s_2}\phi_r)-f(M_r+\mE^{s_2}\psi_r))\big)dr
        \\=&\mE^{s_1}\int_{s_3}^{s_4} \big(P_{r-s_2}f(M_{s_2}+\Gamma_{s_2-r}\mE^{s_1}\phi_r)-P_{r-s_2}f(M_{s_2}+\Gamma_{s_2-r}\mE^{s_1}\psi_r) \\&\qquad-(P_{r-s_2}f(M_{s_2}+\Gamma_{s_2-r}\mE^{s_2}\phi_r)-P_{r-s_2}f(M_{s_2}+\Gamma_{s_2-r}\mE^{s_2}\psi_r))\big)dr.
    \end{align*}
    We apply \eqref{est:semi-group-1} twice, take $\gamma=1$ first and then $\gamma=2$, together with an elementary inequality: for any $x_1,x_2,x_3,x_4\in\mR^{2d}$
    \begin{align*}
        |f(x_1)-f(x_2)-&\big(f(x_3)-f(x_4)\big)|\\&\leq |(x_1-x_2)-(x_3-x_4)| \|f\|_{\bC^1}+ |x_1-x_2||x_1-x_3|\|f\|_{\bC^2},
    \end{align*}
    then get
    \begin{align}\label{est:S1-1}
        |S_1|\leq &\|f\|_{\bC_\bba^\beta}\int_{s_3}^{s_4} (r-s_2)^{\frac{(\beta-1)\wedge0}{\alpha}}\mE^{s_1}|\mE^{s_1}(\phi_r-\psi_r)-\mE^{s_2}(\phi_r-\psi_r)|_\bba dr\nonumber\\&+\|f\|_{\bC_\bba^\beta}\int_{s_3}^{s_4} (r-s_2)^{\frac{(\beta-2)\wedge0}{\alpha}}|\mE^{s_1}(\phi_r-\psi_r)|\mE^{s_1}|(\mE^{s_2}\phi_r-\mE^{s_1}\phi_r)| dr.
    \end{align}
    Now we apply conditional Jensen’s inequality and \eqref{con:S1S2-1} then obtain
    \begin{align}
        \label{est:S1-2}\mE^{s_1}|(\mE^{s_2}\phi_r-\mE^{s_1}\phi_r)|&=\mE^{s_1} |\mE^{s_2}(\phi_r-\mE^{s_1}\phi_r)|\leq \mE^{s_1}|(\phi_r-\mE^{s_1}\phi_r)|\leq C_\phi|r-s_1|^{\gamma_1}.
    \end{align}
    Likewise,
    \begin{align}
        \label{est:S1-3}
        \mE^{s_1}|\mE^{s_1}(\phi_r-\psi_r)-\mE^{s_2}(\phi_r-\psi_r)|=&\mE^{s_1}|\mE^{s_1}\big((\phi_r-\psi_r)-\mE^{s_2}(\phi_r-\psi_r)\big)|
        \nonumber\\\leq& \mE^{s_1}|\big((\phi_r-\psi_r)-\mE^{s_1}(\phi_r-\psi_r)\big)|.
    \end{align}
    Putting \eqref{est:S1-1}, \eqref{est:S1-2}, \eqref{est:S1-3} together with the property of conditional expectation for $\sG\subset\cF_{s_1}$ and applying Minkowski inequality we get,
    \begin{align}
        \label{est:S-1-final}
          \|S_1\|_{L_p(\Omega)|\sG}\leq& N \|f\|_{\bC_\bba^{\beta}}\int_{s_3}^{s_4} \big((r-s_2)^{\frac{(\beta-1)}{\alpha}}|s_4-s_1|^{\gamma_2}[(\phi-\psi)]_{\sC_p^{\gamma_2}|\sG,[S,T]}\nonumber\\ &\quad\qquad\qquad+ (r-s_2)^{\frac{(\beta-2)}{\alpha}}\|\phi-\psi\|_{C_p^0|\sG,[S,T]}|r-s_1|^{\gamma_1}\big)dr  \nonumber\\ \leq &N \|f\|_{\bC_\bba^{\beta}}  \big((s_4-s_2)^{\frac{(\beta-1)}{\alpha}+\gamma_2+1}[(\phi-\psi)]_{\sC_p^{\gamma_2}|\sG,[S,T]}\nonumber\\ &\quad\qquad\qquad+ (s_3-s_2)^{\frac{(\beta-2)}{\alpha}}\|\phi-\psi\|_{C_p^0|\sG,[S,T]}|s_4-s_1|^{\gamma_1}(s_4-s_3)\big)
          \nonumber\\ \leq &N \|f\|_{\bC_\bba^{\beta}} \big((t-s)^{\frac{(\beta-1)}{\alpha}+\gamma_2+1}[(\phi-\psi)]_{\sC_p^{\gamma_2}|\sG}\nonumber\\ &\quad\qquad\qquad+(t-s)^{\frac{(\beta-2)}{\alpha}+1+\gamma_1}\|\phi-\psi\|_{C_p^0|\sG,[S,T]}\big).
    \end{align}
    In the last inequality we applies the fact that $s_3-s_2=s_4-s_3=\frac{t-s}{2}$ and $s_4-s_1=(u-s)+(t-s)=\frac{3(t-s)}{2}$. We get the same estimate for $S_2$ therefore the following holds
    \begin{align}
        \label{est:deltaA}
     \|\mE^{s-(t-s)}\delta A_{s,u,t}\|_{L_p(\Omega)|\sG} &\leq  N \|f\|_{\bC_\bba^{\beta}} \big((t-s)^{\frac{(\beta-1)}{\alpha}+\gamma_2+1}[(\phi-\psi)]_{\sC_p^{\gamma_2}|\sG}\nonumber\\ &\quad\qquad\qquad+(t-s)^{\frac{(\beta-2)}{\alpha}+1+\gamma_1}\|\phi-\psi\|_{C_p^0|\sG,[S,T]}\big).
    \end{align}
Moreover, according to our assumption \eqref{con:constants}, $\frac{(\beta-2)}{\alpha}+1+\gamma_1>1$ and $\frac{(\beta-1)}{\alpha}+\gamma_2+1>1$.    It implies that Lemma \ref{lem:SSL} (2) holds with $K_2= N \|f\|_{\bC_\bba^{\beta}}([(\phi-\psi)]_{\sC_p^{\gamma_2}|\sG}+\|\phi-\psi\|_{C_p^0|\sG,[S,T]})$.

    Let us move to verify Lemma \ref{lem:SSL} (3).  Denote $t_i=s+\frac{(t-s)}{m}i$ for $m\in\mN, i:=\{0,1,\ldots,m\}$ and $(s,t)\in[S,T]_\leq^2$. Since $t_{i}-(t_{i+1}-t_{i})=t_{i-1}$, we have
    \begin{align}
        \label{est:A}\cA_t-\cA_s-\sum_{i=1}^{m-1}&A_{s+i\frac{t-s}{m},s+(i+1)\frac{t-s}{m}} \nonumber\\ =&\cA_t-\cA_s-\sum_{i=1}^{m-1}A_{t_{i},t_{i+1}}
        \nonumber\\\leq&\sum_{i=1}^{m-1}\int_{t_{i}}^{t_{i+1}}|f(M_r+\phi_r)-P_{r-t_{i-1}}f(M_{t_{i-1}}+\Gamma_{t_{i-1}-r}\mE^{t_{i-1}}\phi_r)|dr
         \nonumber\\&+\sum_{i=1}^{m-1}\int_{t_{i}}^{t_{i+1}}|f(M_r+\psi_r)-P_{r-t_{i-1}}f(M_{t_{i-1}}+\Gamma_{t_{i-1}-r}\mE^{t_{i-1}}\psi_r)|dr
         \nonumber\\&+\int_{t_{0}}^{t_{1}}|f(M_r+\phi_r)|+|f(M_r+\psi_r)|dr=:S_3+S_4+S_5.
    \end{align}
    For $S_3$ by triangle inequality and \eqref{est:semi-group-2} with taking $\mu=\frac{\beta}{\alpha}$ we can write
    \begin{align*}
       |f(&M_r+\phi_r)-P_{r-t_{i-1}}f(M_{t_{i-1}}+\Gamma_{t_{i-1}-r}\mE^{t_{i-1}}\phi_r)| \\\lesssim&  |f(M_r+\phi_r)-f(\Gamma_{r-t_{i-1}}M_{t_{i-1}}+\phi_r)| \\&+|f(\Gamma_{r-t_{i-1}}M_{t_{i-1}}+\phi_r)-f(\Gamma_{r-t_{i-1}}M_{t_{i-1}}+\mE^{t_{i-1}}\phi_r)|
       \\&+ |\Gamma_{r-t_{i-1}}f(M_{t_{i-1}}+\Gamma_{t_{i-1}-r}\mE^{t_{i-1}}\phi_r)-P_{r-t_{i-1}}f(M_{t_{i-1}}+\Gamma_{t_{i-1}-r}\mE^{t_{i-1}}\phi_r)|
       \\\lesssim& \|f\|_{\bC_\bba^{\beta}}|M_r-\Gamma_{r-t_{i-1}}M_{t_{i-1}}|_\bba^{\beta}+
        \|f\|_{\bC^1}|\phi_r-\mE^{t_{i-1}}\phi_r|+\|f\|_{\bC_\bba^{{(1+\alpha)\beta}}}|r-t_{i-1}|^{\frac{\beta}{\alpha}-\epsilon}.
    \end{align*}
    We further following from   \eqref{con:S1S2-1} and \eqref{est:moment-M} then  have
    \begin{align}\label{est:S3-m}
       \| S_3\|_{L_1(\Omega)}\leq  &N\|f\|_{\bC_\bba^{\beta}}|t-s|\sup_{i=0,1,\ldots,m}|t_{i+1}-{t_{i-1}}|^{\frac{\beta}{\alpha}}+N\|f\|_{\bC^{1}}|t-s|\sup_{i=0,1,\ldots,m}|t_{i+1}-{t_{i-1}}|^{\gamma_1}\nonumber\\&+N\|f\|_{\bC_\bba^{(1+\alpha)\beta}}|t-s|(\sup_{i}|t_{i+1}-t_{i-1}|^{\frac{\beta}{\alpha}-\epsilon})\nonumber\\\lesssim&N(\|f\|_{\bC^{1}}+\|f\|_{\bC_\bba^{(1+\alpha)\beta}})|t-s|m^{-\big(\frac{\beta}{\alpha}\wedge \gamma_1\big)+\epsilon}.
    \end{align}
    In a similar manner we get
    \begin{align}\label{est:S4-m}
          \| S_4\|_{L_1(\Omega)}\leq  &N(\|f\|_{\bC^{1}}+\|f\|_{\bC_\bba^{(1+\alpha)\beta}})|t-s|m^{-\big(\frac{\beta}{\alpha}\wedge \gamma_3\big)+\epsilon}.
    \end{align}
    For  $S_5$ we easily get
    \begin{align}\label{est:S5-m}
          \| S_5\|_{L^1(\Omega)}\leq  &N\|f\|_\infty m^{-1}.
    \end{align}
    Plug \eqref{est:S3-m}, \eqref{est:S4-m} and \eqref{est:S5-m}  into \eqref{est:A} we can conclude that
    \begin{align*}
        \cA_t-\cA_s=\lim_{m\rightarrow\infty}\sum_{i=1}^{m-1}A_{s+i\frac{t-s}{m},s+(i+1)\frac{t-s}{m}} \quad\text{ in } L_1(\Omega).
    \end{align*}
    Hence Lemma \ref{lem:SSL} (3) is verified.

    Finally all of the conditions from Lemma \ref{lem:SSL} are verified and \eqref{est:S1S2} follows from \eqref{con:SSL-est}.
\end{proof}

\begin{corollary}
    \label{cor:S1S2}
For any $0<\epsilon$ and $\beta\in(1-\frac{\alpha}{2},(1+\alpha)(\alpha-1))$ there exists a constant $N=N(\alpha,p,\beta,d,\|b\|_{\bC_\bba^{\beta}},\epsilon)$ independent of $n$ so that for any $(s,t)\in[0,1]_\leq^2$ and any $\sigma$-algebra $\sG\subset\cF_{s}$ the following holds
     \begin{align}
         \label{est:cor-S1} \|S_{s,t}^{b,n,1}\|_{L_p(\Omega)|\sG} \leq& N|t-s|^{\frac{(\beta-1)}{\alpha}+\frac{3}{2}}[Y-Y^n]_{C_p^{\frac{1}{2}}|\sG,[s,t]}+N |t-s|^{\frac{(\beta-1)}{\alpha}+1}[Y-Y^n]_{C_p^0|\sG,[s,t]}; \\
         \label{est:cor-S2} \|S_{s,t}^{b,n,2}\|_{L_p(\Omega)|\sG}\leq &N |t-s|^{\frac{(\beta-1)}{\alpha}+\frac{3}{2}}n^{-(\frac{1}{2}+\frac{\beta}{\alpha}\wedge\frac{1}{2})+\epsilon}+N|t-s|^{\frac{(\beta-1)}{\alpha}+1}n^{-1}.
     \end{align}
\end{corollary}
\begin{proof}
 We apply Lemma \ref{lem:auxli-S1S2} to show the desired results.

    For \eqref{est:cor-S1}, we take $\phi=Y$ and $\psi=Y^n$, $\gamma_1=\gamma_3=1+\frac{\beta}{\alpha}-\epsilon,\gamma_2=\frac{1}{2}$ (choose $\epsilon$ to be sufficiently small so that $\beta>1-\frac{\alpha}{2}+\epsilon\alpha$)  in \eqref{con:S1S2-1} and  \eqref{con:S1S2-2} correspondingly. Hence the condition \eqref{con:S1S2-1} of Lemma \ref{lem:auxli-S1S2}  is verified by \eqref{est:apaiori-1} and \eqref{est:apaiori-2} with taking $p=1$ therein. We see that \eqref{est:cor-S1} holds because of \eqref{est:S1S2}.

    For  \eqref{est:cor-S2}, we take $\phi_t=Y^n_t$ and $\psi_t=\Gamma_{t-k_n(t)}Y_{k_n(t)}^n$, and $\gamma_1,\gamma_2,\gamma_3, \epsilon$ the same as above. The condition \eqref{con:S1S2-1} holds following from \eqref{est:apaiori-2} with taking $p=1$. For \eqref{con:S1S2-2}, if $0\leq k_n(t)\leq s < t \leq 1$ then
    \begin{align*}
    \mE^s|\Gamma_{t-k_n(t)}Y_{k_n(t)}^n-\mE^s(\Gamma_{t-k_n(t)}Y_{k_n(t)}^n)|=\mE^s|\Gamma_{t-k_n(t)}(|Y_{k_n(t)}^n-\mE^sY_{k_n(t)}^n)|=0;
    \end{align*}
    if $0 \leq s< k_n(t)\leq  t \leq 1$ then by \eqref{est:apaiori-2} and \eqref{0222:00}
    \begin{align*}
        \mE^s|\Gamma_{t-k_n(t)}Y_{k_n(t)}^n-\mE^s(\Gamma_{t-k_n(t)}Y_{k_n(t)}^n)|=&\mE^s|\Gamma_{t-k_n(t)}(|Y_{k_n(t)}^n-\mE^sY_{k_n(t)}^n)|\\\lesssim&\mE^s|Y_{k_n(t)}^n-\mE^sY_{k_n(t)}^n|\leq N|k_n(t)-s|\leq N|t-s|,
    \end{align*}
    which shows that \eqref{con:S1S2-2} holds. Hence from \eqref{est:S1S2} we get that
    \begin{align*}
        \|S_{s,t}^{b,n,2}\|_{L_p(\Omega)|\sG}\leq &N|t-s|^{\frac{(\beta-1)}{\alpha}+\frac{3}{2}}[Y^n-\Gamma_{\cdot-k_n(\cdot)}Y^n_{k_n(\cdot)}]_{\sC_p^{\frac{1}{2}}|\sG,[s,t]}\\&+N |t-s|^{\frac{(\beta-1)}{\alpha}+1}[Y^n-\Gamma_{\cdot-k_n(\cdot)}Y^n_{k_n(\cdot)}]_{C_p^0|\sG,[s,t]}.
    \end{align*}
    In the end we get \eqref{est:cor-S2} after plugging \eqref{est:Yn-Yknt-1} and \eqref{est:Yn-Yknt-2} into above inequality.  The proof completes.
\end{proof}
\subsection{Estimates of $S^{b,n,3}$}
\begin{lemma}
    \label{lemm:est-S3} Let $\phi:[0,1]\times\Omega\rightarrow\mR^{2d}$ be a bounded measurable function. Let    $f\in\mC_\bba ^{(1+\alpha)\beta}$ for $\beta\in (0,\alpha-1)$. Further assume
    \begin{enumerate}
        \item there exist constants $\gamma_1,\gamma_\phi>0$ so that
        \begin{align}
        \label{con:constant-S3-0}
        \gamma_1&>\frac{1}{2}+\frac{1}{\alpha}+\frac{\beta}{\alpha}\wedge\frac{1}{2}-\frac{\beta}{\alpha};\\
           \label{con:constant-S3-1}
           \mE^s|\mE^t\phi_r-\mE^s\phi_r|&\leq \gamma_\phi|r-s|^{\gamma_1}, \quad  (s,r)\in[0,1]_\leq^2, t\in[s,1];
        \end{align}
        \item  $\phi_{k_n(t)}$ is $\cF_{(k_n(t)-n^{-1})\vee0}$ measurable,  $n\in\mN$;
        \item $\phi\equiv0$ or $\beta>0$.
    \end{enumerate}
    Then for any $0<\epsilon<\frac{1}{2}$ there exists a constant $N=N(\alpha,\beta,d,\gamma_1,\epsilon)$ independent of $n$ so that for any $(S,T)\in [0,1]_\leq^2$ and any $\sigma$-algebra $\sG\subset\cF_{k_n(s)}$ the following holds
    \begin{align}
        \label{est:S3-final}
         \| \int_S^T &f(M_r+\Gamma_{r-k_n(r)}\phi_{k_n(r)})-\Gamma_{r-k_n(r)}f(M_{k_n(r)}+\phi_{k_n(r)})dr\|_{L_2(\Omega)|\sG}\nonumber\\&\leq N \|f\|_{\mC_\bba^{(1+\alpha)\beta}}(1+\gamma_\phi)n^{-(\frac{1}{2}+\frac{\beta}{\alpha}\wedge\frac{1}{2})+2\epsilon}|T-S|^{\frac{1}{2}+\epsilon}.
    \end{align}
\end{lemma}
\begin{proof}
Denote for any $(s,t)\in\Delta_{[S,T]}$, $(S,T)\in[0,1]_\leq^2$,
    \begin{align*}
        A_{s,t}:=&\mE^{s-(t-s)}\int_s^tf(M_r+\Gamma_{r-k_n(r)}\mE^{s-(t-s)}\phi_{k_n(r)})-\Gamma_{r-k_n(r)}f(M_{k_n(r)}+\mE^{s-(t-s)}\phi_{k_n(r)})dr,\\
        \cA_{t}:=&\int_0^tf(M_r+\Gamma_{r-k_n(r)}\phi_{k_n(r)})-\Gamma_{r-k_n(r)}f(M_{k_n(r)}+\phi_{k_n(r)})dr.
    \end{align*}
   We want to applyLemma \ref{lem:SSL} to such $A$ and $\cA$. Now we verify its conditions.
 We start with Lemma \ref{lem:SSL} (1). We analyze it in two cases.
   \begin{itemize}
       \item[Case I.] $t\in[s,s+\frac{2}{n}]$.
   \end{itemize}
   Following from \eqref{est:moment-M},  \eqref{prop:Markov}, 
   and the fact that $|t-s|\lesssim n^{-1}$ we have
   \begin{align*}
     &  \|A_{s,t}\|_{L_2(\Omega)|\sG}\\&\leq \int_s^t\|f(M_r+\Gamma_{r-k_n(r)}\mE^{s-(t-s)}\phi_{k_n(r)})-\Gamma_{r-k_n(r)}f(M_{k_n(r)}+\mE^{s-(t-s)}\phi_{k_n(r)})\|_{L_2(\Omega)|\sG}dr
       \\&\leq \|f\|_{\mC_\bba^\beta}\int_s^t\big\||M_r-\Gamma_{r-k_n(r)}M_{k_n(r)}
       |_\bba^\beta\wedge1\big\|_{L_2(\Omega)|\sG}dr
         \\&\leq \|f\|_{\mC_\bba^\beta} \int_s^t |r-k_n(r)|^{\frac{\beta}{\alpha}\wedge\frac{1}{2}-\epsilon}dr
           \\&\leq \|f\|_{\mC_\bba^\beta} n^{-(\frac{\beta}{\alpha+1}\wedge\frac{1}{2})+\epsilon}
           |t-s| \\&\leq \|f\|_{\mC_\bba^\beta} n^{-(\frac{1}{2}+\frac{\beta}{\alpha}\wedge\frac{1}{2})+2\epsilon}
           |t-s|^{\frac{1}{2}+\epsilon}.
   \end{align*}
\begin{itemize}
    \item[Case II.] $t\in[s+\frac{2}{n},1]$.
\end{itemize}
In this case we know that $k_n(r)\geq s-(t-s)$ for $r\geq s$, and also $\frac{t-s}{2}\leq k_n(r)- (s-(t-s))$. Following from \eqref{prop:Markov} and \eqref{eq:inde-incre} and the property of $(\Gamma_t)_{t\geq0}$ we can write
\begin{align*}
    A_{s,t}=&\int_s^tP_{r- (s-(t-s))}f(M_{s-(t-s)}+\Gamma_{s-(t-s)-r}\Gamma_{r-k_n(r)}\mE^{s-(t-s)}\phi_{k_n(r)})\\&\qquad\qquad-P_{k_n(r)- (s-(t-s))}\Gamma_{r-k_n(r)}f(M_{s-(t-s)}+\Gamma_{s-(t-s)-k_n(r)}\mE^{s-(t-s)}\phi_{k_n(r)})dr
 \\   =&\int_s^t\big(P_{r- (s-(t-s))}-P_{k_n(r)- (s-(t-s))}\Gamma_{r-k_n(r)})\big)f(M_{s-(t-s)}+\Gamma_{s-(t-s)-k_n(r)}\mE^{s-(t-s)}\phi_{k_n(r)})dr
\\   =&:I.
\end{align*}
Then applying \eqref{est:semi-group-2} with $ \gamma=0, \mu:=(\frac{1}{2}+\frac{\beta}{\alpha}\wedge\frac{1}{2})-\epsilon$   gives us
\begin{align*}
    |I|\leq &N\|f\|_{\mC_\bba^{(1+\alpha)\beta}}\int_s^tn^{-(\frac{1}{2}+\frac{\beta}{\alpha}\wedge\frac{1}{2})+2\epsilon}|t-s|^{\frac{\beta}{\alpha}\wedge(1-\epsilon)-(\frac{\beta}{\alpha}\wedge\frac{1}{2})+\epsilon}dr
    \\\leq&N\|f\|_{\mC_\bba^{(1+\alpha)\beta}} n^{-(\frac{1}{2}+\frac{\beta}{\alpha}\wedge\frac{1}{2})+2\epsilon}|t-s|^{\frac{1}{2}+\epsilon},
    \end{align*}
   which yields that
    \begin{align*}
        \|A_{s,t}\|_{L_2(\Omega)|\sG}\leq N\|f\|_{\mC_\bba^{(1+\alpha)\beta}} n^{-(\frac{1}{2}+\frac{\beta}{\alpha}\wedge\frac{1}{2})+2\epsilon}|t-s|^{\frac{1}{2}+\epsilon}.
    \end{align*}

  Combining those two cases together we have that  Lemma \ref{lem:SSL} (1) holds with taking $K_1=N\|f\|_{\mC_\bba^{(1+\alpha)\beta}} n^{-(\frac{1}{2}+\frac{\beta}{\alpha}\wedge\frac{1}{2})+2\epsilon}.$

  Now we check  Lemma \ref{lem:SSL} (2). As in Lemma \ref{lem:auxli-S1S2},  let $(s,t)\in\Delta_{[S,T]}$ and $u:=\frac{t+s}{2}$, for convenience also denote $s_1:=s-(t-s)$, $s_2:=s-(u-s)$, $s_3:=s$,  $s_4:=u$, $s_5:=t$ so that one has $s_1\leq s_2\leq s_3\leq  s_4\leq s_5$.  Observe  $\mE^{s-(t-s)}\delta A_{s_3,s_4,s_5}=\mE^{s_1}\delta A_{s_3,s_4,s_5}$. We can write
  \begin{align*}
    \mE^{s_1}\delta A_{s_3,s_4,s_5}=&\mE^{s_1} \mE^{s_2}\int_{s_3}^{s_4}f(M_r+ \Gamma_{r-k_n(r)} \mE^{s_1}\phi_{k_n(r)})- \Gamma_{r-k_n(r)} f(M_{k_n(r)}+ \mE^{s_1}\phi_{k_n(r)})\\&\qquad\qquad-f(M_r+ \Gamma_{r-k_n(r)} \mE^{s_2}\phi_{k_n(r)})+ \Gamma_{r-k_n(r)} f(M_{k_n(r)}+ \mE^{s_2}\phi_{k_n(r)})dr
    \\&+\mE^{s_1} \mE^{s_3}\int_{s_4}^{s_5}f(M_r+ \Gamma_{r-k_n(r)} \mE^{s_1}\phi_{k_n(r)})- \Gamma_{r-k_n(r)} f(M_{k_n(r)}+ \mE^{s_1}\phi_{k_n(r)})\\&\qquad\qquad-f(M_r+ \Gamma_{r-k_n(r)} \mE^{s_3}\phi_{k_n(r)})+ \Gamma_{r-k_n(r)} f(M_{k_n(r)}+ \mE^{s_3}\phi_{k_n(r)})dr\\=&:I_1+I_2.
  \end{align*}
  We can notice that those two terms can be estimated in the same way so we only need to estimate one of them and the other one follows similarly. Let us estimate $I_1$. Again we analyse in two cases.
  \begin{itemize}
      \item[Case I.] $|t-s|\geq \frac{4}{n}$.
  \end{itemize}
  For $r\geq s_3$ we have $0<\frac{t-s}{4}\leq s_3-n^{-1}-s_2\leq k_n(r)-s_2$ and
  \begin{align*}
      I_1=&\mE^{s_1} \int_{s_3}^{s_4}P_{r-s_2}f(M_{s_2}+ \Gamma_{s_2-r}\Gamma_{r-k_n(r)} \mE^{s_1}\phi_{k_n(r)})\\&\qquad\qquad- P_{k_n(r)-s_2}\Gamma_{r-k_n(r)} f(M_{s_2}+ \Gamma_{s_2-k_n(r)}\mE^{s_1}\phi_{k_n(r)})\\&\qquad\qquad-P_{r-s_2}f(M_{s_2}+ \Gamma_{s_2-r}\Gamma_{r-k_n(r)} \mE^{s_2}\phi_{k_n(r)})\\&\qquad\qquad+ P_{k_n(r)-s_2}\Gamma_{r-k_n(r)} f(M_{s_2}+ \Gamma_{s_2-k_n(r)}\mE^{s_2}\phi_{k_n(r)})dr
      \\=&\mE^{s_1} \int_{s_3}^{s_4}(P_{r-s_2}-P_{k_n(r)-s_2}\Gamma_{r-k_n(r)} )f(M_{s_2}+ \Gamma_{s_2-k_n(r)}\mE^{s_1}\phi_{k_n(r)})\\&\qquad\qquad-(P_{r-s_2}-P_{k_n(r)-s_2}\Gamma_{r-k_n(r)} )f(M_{s_2}+ \Gamma_{s_2-k_n(r)}\mE^{s_2}\phi_{k_n(r)})dr,
  \end{align*}
  inside the above we also applied the fact that $\Gamma_{s_2-r}\Gamma_{r-k_n(r)}=\Gamma_{s_2-k_n(r)}$.
  We apply \eqref{est:semi-group-2} with $ \gamma=1, \mu:=(\frac{1}{2}+\frac{\beta}{\alpha}\wedge\frac{1}{2})-\epsilon$ together with \eqref{con:constant-S3-1} and taking $\epsilon':=\gamma_1-(\frac{1}{2}+\frac{\beta}{\alpha}\wedge\frac{1}{2}+\frac{1}{\alpha}-\frac{\beta}{\alpha})>0$ by \eqref{con:constant-S3-0}  then get
  \begin{align*}
      |I_1|\leq &\int_{s_3}^{s_4}\|P_{r-s_2}-P_{k_n(r)-s_2}\Gamma_{r-k_n(r)}f \|_{\bC^1}|\Gamma_{s_2-k_n(r)}(\mE^{s_1}-\mE^{s_2})\phi_{k_n(r)})| dr
      \\\leq &N\|f\|_{\mC_\bba^{(1+\alpha)\beta}}\int_{s_3}^{s_4} n^{-(\frac{1}{2}+\frac{\beta}{\alpha}\wedge\frac{1}{2})+\epsilon}|t-s|^{-(\frac{1}{2}+\frac{\beta}{\alpha}\wedge\frac{1}{2}+\frac{1}{\alpha}-\frac{\beta}{\alpha})}\gamma_\phi|s_2-s_1|^{\gamma_1}dr
      \\\leq &N\|f\|_{\mC_\bba^{(1+\alpha)\beta}}\gamma_\phi n^{-(\frac{1}{2}+\frac{\beta}{\alpha}\wedge\frac{1}{2})+\epsilon}|t-s|^{1+\epsilon'}.
  \end{align*}
  As we said before, similarly we also have
  \begin{align*}
  |I_2|\leq     N\|f\|_{\mC_\bba^{(1+\alpha)\beta}}\gamma_\phi n^{-(\frac{1}{2}+\frac{\beta}{\alpha}\wedge\frac{1}{2})+\epsilon}|t-s|^{1+\epsilon'}.
  \end{align*}
  Hence we get
  \begin{align*}
       \|\mE^{s_1}\delta A_{s_3,s_4,s_5}\|_{L_2(\Omega)|\sG}\leq N\|f\|_{\mC_\bba^{(1+\alpha)\beta}}\gamma_\phi n^{-(\frac{1}{2}+\frac{\beta}{\alpha}\wedge\frac{1}{2})+\epsilon}|t-s|^{1+\epsilon'}.
  \end{align*}
  \begin{itemize}
      \item[Case II.] $|t-s|\leq \frac{4}{n}$.
  \end{itemize}
 If $r\in[k_n(s_1),k_n(s_1)+\frac{2}{n}]$, by assumption (2) we know that $\phi_{k_n(r)}$ is $\cF_{(k_n(r)-n^{-1})\vee0}$-measurable therefore
  $\mE^{s_1}\phi_{k_n(r)}=\mE^{s_2}\phi_{k_n(r)}=\phi_{k_n(r)}$, it yields that $I_1=0$. So we only need to care the interval $r\in[k_n(s_1)+\frac{2}{n},1]$ in which $k_n(r)-n^{-1}\geq s_1$ and we can write
  \begin{align*}
      &I_1\\=&\mE^{s_1}\int_{[s_3,s_4]\cap[k_n(s_1)+\frac{2}{n},1]}\mE^{k_n(r)-n^{-1}}\big(f(M_r+ \Gamma_{r-k_n(r)}\mE^{s_1}\phi_{k_n(r)})- \Gamma_{r-k_n(r)} f(M_{k_n(r)}+ \mE^{s_1}\phi_{k_n(r)})\\&\qquad\qquad\qquad-f(M_r+  \Gamma_{r-k_n(r)}\mE^{s_2}\phi_{k_n(r)})+ \Gamma_{r-k_n(r)} f(M_{k_n(r)}+ \mE^{s_2}\phi_{k_n(r)})\big)dr\\=&\mE^{s_1}\int_{[s_3,s_4]\cap[k_n(s_1)+\frac{2}{n},1]} P_{r-k_n(r)+n^{-1}} f(M_{k_n(r)-n^{-1}}+ \Gamma_{k_n(r)-n^{-1}-r} \Gamma_{r-k_n(r)}\mE^{s_1}\phi_{k_n(r)})
      \\&\qquad\qquad\qquad-P_{r-k_n(r)+n^{-1}} f(M_{k_n(r)-n^{-1}}+ \Gamma_{k_n(r)-n^{-1}-r} \Gamma_{r-k_n(r)}\mE^{s_2}\phi_{k_n(r)}) \\&\qquad\qquad\qquad+P_{n^{-1}}\Gamma_{r-k_n(r)}f(M_{k_n(r)-n^{-1}}+ \Gamma_{{k_n(r)-n^{-1}}} \mE^{s_2}\phi_{k_n(r)})\\&\qquad\qquad\qquad-P_{n^{-1}}\Gamma_{r-k_n(r)} f(M_{k_n(r)-n^{-1}}+\Gamma_{{k_n(r)-n^{-1}}}  \mE^{s_1}\phi_{k_n(r)})
     dr.
  \end{align*}
  In the second equality we applied the fact that $\phi_{k_n(r)}$ is $\cF_{k_n(r)-n^{-1}}$-measurable.  By  \eqref{est:semi-group-2} with $\gamma=1$ and \eqref{con:constant-S3-0}, \eqref{con:constant-S3-1} we get
  \begin{align*}
      |I_1|\leq &\int_{[s_3,s_4]\cap[k_n(s_1)+\frac{2}{n},1]} \big(\|P_{r-k_n(r)+n^{-1}} f\|_{\bC^1}+\|P_{n^{-1}}\Gamma_{r-k_n(r)} f\|_{\bC^1}\big)
      |\mE^{s_1}\phi_{k_n(r)}-\mE^{s_2}\phi_{k_n(r)}|dr
      \\\leq&N(\|f\|_{\bC_\bba^\beta}+\|f\|_{\bC_\bba^{(1+\alpha)\beta}})\gamma_\phi\int_{[s_3,s_4]\cap[k_n(s_1)+\frac{2}{n},1]}n^{-\frac{\beta-1}{\alpha}}|s_2-s_1|^{\gamma_1}dr
      \\\leq&N\|f\|_{\bC_\bba^{(1+\alpha)\beta}}\gamma_\phi n^{-\gamma_1-\frac{\beta-1}{\alpha}+\epsilon}|t-s|^{1+\epsilon}
\\\leq&N\|f\|_{\bC_\bba^{(1+\alpha)\beta}}\gamma_\phi n^{-(\frac{1}{2}+\frac{\beta}{\alpha}\wedge\frac{1}{2})+\epsilon}|t-s|^{1+\epsilon},
  \end{align*}
  where the penultimate one follows from the fact that $|t-s|\leq\frac{4}{n}$ and the last one is from \eqref{con:constant-S3-0}.

  Now we can conclude that  Lemma \ref{lem:SSL} (2) holds with $K_2=N\|f\|_{\bC_\bba^{(1+\alpha)\beta}}\gamma_\phi n^{-(\frac{1}{2}+\frac{\beta}{\alpha}\wedge\frac{1}{2})+\epsilon}.$

  Lastly we verify  Lemma \ref{lem:SSL} (3).  Denote $t_i=s+\frac{(t-s)}{m}i$ for $m\in\mN, i:=\{0,1,\ldots,m\}$ and $(s,t)\in[S,T]_\leq^2$. Since $t_{i}-(t_{i+1}-t_{i})=t_{i-1}$, we have
    \begin{align}
        \label{est:A-S-3}\|\cA_t-\cA_s-\sum_{i=1}^{m-1}&A_{s+i\frac{t-s}{m},s+(i+1)\frac{t-s}{m}} \|_{L_1(\Omega)}\nonumber\\ =&\|\cA_t-\cA_s-\sum_{i=1}^{m-1}A_{t_{i},t_{i+1}}\|_{L_1(\Omega)}
        \nonumber\\\leq& \|\cA_{t_{1}}-\cA_{t_{0}}\|_{L_1(\Omega)}
    \nonumber\\&  + \big\| \sum_{i=1}^{m-1}\cA_{t_{i+1}}-\cA_{t_{i}}-\mE^{t_i}(\cA_{t_{i+1}}-\cA_{t_{i}})\big\|_{L_2(\Omega)}  \nonumber\\&+\big\| \sum_{i=1}^{m-1}\mE^{t_i}(\cA_{t_{i+1}}-\cA_{t_{i}})-\mE^{t_{i-1}}(\cA_{t_{i+1}}-\cA_{t_{i}})\big\|_{L_2(\Omega)}
         \nonumber\\&+\big\|\sum_{i=1}^{m-1}A_{t_{i},t_{i+1}}-\mE^{t_{i-1}}(\cA_{t_{i+1}}-\cA_{t_{i}})\big\|_{L_1(\Omega)}=:I_3+I_4+I_5+I_6.
    \end{align}
    It is easy to get that
    \begin{align}\label{est:I3}
        I_3\leq 2\|f\|_\infty m^{-1}.
    \end{align}
    For $I_4$ we observe that $(\cA_{t_{i+1}}-\cA_{t_{i}}-\mE^{t_i}(\cA_{t_{i+1}}-\cA_{t_{i}}))_{i=1,\ldots,m-1}$ is a martingale difference sequence with respect to $(\cF_{t_{i+1}})_{i=1,\ldots,m-1}$. By Burkholder-Davis-Gundy inequality we get
    \begin{align}\label{est:I4}
        I_4^2\leq \sum_{i=1}^{m-1}\big\| \cA_{t_{i+1}}-\cA_{t_{i}}-\mE^{t_i}(\cA_{t_{i+1}}-\cA_{t_{i}})\big\|_{L_2(\Omega)}^2 \leq 2\|f\|_\infty^2 m^{-1}.
    \end{align}
    Likewise we get that $(\mE^{t_i}(\cA_{t_{i+1}}-\cA_{t_{i}})-\mE^{t_{i-1}}(\cA_{t_{i+1}}-\cA_{t_{i}}))_{i=1,\ldots,m-1}$ is a martingale difference sequence with respect to $(\cF_{t_{i}})_{i=1,\ldots,m-1}$ so we get
       \begin{align}\label{est:I5}
        I_5^2\leq \sum_{i=1}^{m-1}\big\| \mE^{t_i}(\cA_{t_{i+1}}-\cA_{t_{i}})-\mE^{t_{i-1}}(\cA_{t_{i+1}}-\cA_{t_{i}})\big\|_{L_2(\Omega)}^2 \leq 16\|f\|_\infty^2 m^{-1}.
    \end{align}
    For $I_6$, if $\phi\equiv0$ then $I_6=0$; otherwise by condition (3) and the fact that $ \|f\|_{\bC^\beta}\lesssim  \|f\|_{\bC_\bba^{(1+\alpha)\beta}}$ we have for $\beta>0$, and for $i=1,\ldots, m-1$
    \begin{align*}
       \big \|A_{t_{i},t_{i+1}}&-\mE^{t_{i-1}}(\cA_{t_{i+1}}-\cA_{t_{i}})\big\|_{L_1(\Omega)}\\\leq& \big \|\int_{t_{i}}^{t_{i+1}}f(M_r+ \Gamma_{r-k_n(r)}\phi_{k_n(r)})-  f(M_{r}+ \Gamma_{r-k_n(r)}\mE^{t_{i-1}}\phi_{k_n(r)})dr\big\|_{L_1(\Omega)}\\&+\big \|\int_{t_{i}}^{t_{i+1}}\Gamma_{r-k_n(r)} f(M_{k_n(r)}+ \phi_{k_n(r)})-\Gamma_{r-k_n(r)} f(M_{k_n(r)}+ \mE^{t_{i-1}}\phi_{k_n(r)})dr\big\|_{L_1(\Omega)}
       \\\leq& \|f\|_{\bC^\beta}\int_{t_{i}}^{t_{i+1}}\|\Gamma_{r-k_n(r)}(\phi_{k_n(r)}- \mE^{t_{i-1}}\phi_{k_n(r)})\|_{L_1(\Omega)}^\beta dr
       \\\leq&
      \|f\|_{\bC_\bba^{(1+\alpha)\beta}} \int_{t_{i}}^{t_{i+1}}  \|\phi_{k_n(r)}-  \mE^{t_{i-1}}\phi_{k_n(r)}\big\|_{L_1(\Omega)}^\beta dr
       \\\leq& \|f\|_{\bC_\bba^{(1+\alpha)\beta}}|t_{i+1}-t_i|\gamma_\phi|t_{i+1}-t_i|^{\beta\gamma_1}
       \\\leq& N\|f\|_{\bC_\bba^{(1+\alpha)\beta}} (1+\gamma_\phi)m^{-1-\gamma_1\beta},
    \end{align*}
    in above we applied the fact that if $k_n(r)\leq t_{i-1}$ then $\mE^{t_{i-1}}\phi_{k_n(r)}=\phi_{k_n(r)}$, if $k_n(r)\geq t_{i-1}$ then \eqref{con:constant-S3-1} works  and $k_n(r)- t_{i-1}\leq t_{i+1}-t_{i-1} $. We plug the above estimate into $I_6$ then get
    \begin{align}\label{est:I6}
        I_6\leq N\|f\|_{\bC_\bba^{(1+\alpha)\beta}} (1+\gamma_\phi)m^{-\gamma_1\beta}\mathbf{1}_{\beta>0},
    \end{align}
    We now combine \eqref{est:A-S-3} with \eqref{est:I3}, \eqref{est:I4}, \eqref{est:I5} and \eqref{est:I6} togehter, we obtain
    \begin{align*}
        \|\cA_t-\cA_s-\sum_{i=1}^{m-1}A_{t_{i},t_{i+1}}\|_{L_1(\Omega)}\leq N\|f\|_{\bC_\bba^{(1+\alpha)\beta}}
        (1+\gamma_\phi)m^{-\gamma_1\beta}\mathbf{1}_{\beta>0}+N\|f\|_\infty m^{-1/2}.
    \end{align*}
    which implies that Lemma \ref{lem:SSL} (3) holds.

    In the end all of the conditions from Lemma \ref{lem:SSL} are verified therefore \eqref{est:S3-final} holds by \eqref{con:SSL-est}.
\end{proof}
\begin{corollary}
    \label{cor:est-S3}For any $\beta\in(1-\frac{\alpha}{2},(\alpha-1)(1+\alpha))$ and $0<\epsilon$ there exists a constant $N=N(\alpha,\beta,d,\|b\|_{\mC_\bba^{\beta}},\epsilon)$ independent of $n$ so that for any $(s,t)\in[0,1]_\leq^2$ and any $\sigma$-algebra $\sG\subset\cF_{s}$ the following holds
     \begin{align}
         \label{est:cor-S3} \|S_{s,t}^{b,n,3}\|_{L_2(\Omega)|\sG} \leq& N|t-s|^{\frac{1}{2}+\epsilon}n^{-(\frac{1}{2}+\frac{\beta}{(1+\alpha)\alpha}\wedge\frac{1}{2})+2\epsilon}.
     \end{align}
\end{corollary}
\begin{proof}
    We apply Lemma \ref{lemm:est-S3} with taking $\phi=Y^n$. 
     Choose $\delta$ to be sufficiently small so that $\beta>1-\frac{\alpha}{2}+\delta\alpha$.

    Let us check the conditions of Lemma \ref{lemm:est-S3}. For any $(s,t)\in[0,1]_\leq^2, r\in[s,t]$, we have by \eqref{est:apaiori-2} with $q=1$ we get
    \begin{align*}
        \mE^s|\mE^tY^n_r-\mE^sY^n_r|=    \mE^s|\mE^t(\mE^tY^n_r-\mE^sY^n_r)| \leq \mE^s|\mE^sY^n_r-Y^n_r|\lesssim |r-s|^{1-\delta+\frac{\beta}{\alpha}}.
    \end{align*}
    Therefore \eqref{con:constant-S3-1} holds with $\gamma_1=1-\delta+\frac{\beta}{\alpha}$ which also fulfills \eqref{con:constant-S3-0}; indeed, notice that we take $\beta>1-\frac{\alpha}{2}+\delta\alpha$, which implies that   \eqref{con:constant-S3-0} holds with
    \begin{align*}
        \gamma_1=1-\delta+\frac{\beta}{\alpha}>\frac{1}{2}+\frac{1}{\alpha}\geq\frac{1}{2}+\frac{1}{\alpha}+\frac{\beta}{\alpha}\wedge\frac{1}{2}-\frac{\beta\wedge1}{\alpha}.
    \end{align*}
    We also see that $Y_{k_n(t)}^n$ is $\cF_{k_n(t)\vee0}$-measurable for $t\in[0,1]$.  Now all of the conditions of Lemma \ref{lemm:est-S3} are verified, hence \eqref{est:cor-S3} follows directly from \eqref{est:S3-final}.
\end{proof}
\subsection{Crucial conditional quadrature estimate}
\begin{theorem}
  \label{thm:qua}For any $\beta\in(1-\frac{\alpha}{2},(\alpha-1)(1+\alpha))$ and any $0<\epsilon<\frac{1}{2}+\frac{(\beta-1)}{\alpha}$ there exists a constant $N=N(\alpha,\beta,d,\|b\|_{\mC_\bba^{\beta}},\epsilon)$ independent of $n$ so that for any $n\in\mN$ and $k=0,1,\ldots,n$,  $(s,t)\in[\frac{k}{n},1]_\leq^2$ the following holds
     \begin{align}
         \label{est:qua} &\|\int_s^tb(Z_r)-\Gamma_{r-k_n(r)}b(Z_{k_n(r)}^n)dr\|_{L_2(\Omega)|\cF_{\frac{k}{n}}} \nonumber\\& \leq N|t-s|^{\frac{1}{2}+\epsilon}(n^{-(\frac{1}{2}+\frac{\beta}{(1+\alpha)\alpha}\wedge\frac{1}{2})+2\epsilon}+[Y-Y^n]_{C_p^{\frac{1}{2}}|\cF_{\frac{k}{n}},[s,t]}+[Y-Y^n]_{C_p^0|\cF_{\frac{k}{n}},[s,t]}).
     \end{align}
\end{theorem}
\begin{proof}
   Fix  $n\in\mN$ and $k$,  $(s,t)\in[\frac{k}{n},1]_\leq^2$. Following from \eqref{est:cor-S1}, \eqref{est:cor-S2} and \eqref{est:cor-S3} with taking $\sG=\cF_{\frac{k}{n}}$  and noticing that $\frac{k}{n}\leq k_n(s)$ hence $\cF_{\frac{k}{n}}\subset \cF_{k_n(s)}$, we get from \eqref{eq:main-error} that  \eqref{est:qua} holds considering $\frac{3}{2}+\frac{(\beta-1)}{\alpha}, 1+\frac{\beta-1}{2}>\frac{1}{2}+\epsilon$ and $\frac{1}{2}+\frac{\beta}{(1+\alpha)\alpha}\wedge\frac{1}{2}-2\epsilon<1$. The proof completes.
\end{proof}
\section{Proof of the main result} \label{sec:fin-proof}
\begin{proof}[Proof of Theorem \ref{thm:main-rough}]
    We follow the  idea from \cite[Proof of Theorem 2.2]{BDG-Levy}.

Notice that under Assumption \ref{ass:main1},
   if we denote $\beta':=\frac{\beta}{1+\alpha}\in(0,\alpha-1)$, then $\|b\|_{\bC_\bba^{(1+\alpha)\beta'}}=\|b\|_{\bC_\bba^{\beta}}$ is finite.
    \begin{align*}
       \rho:=(\frac{1}{2}+\frac{\beta'}{\alpha}\wedge \frac{1}{2})-\epsilon.
    \end{align*}

    In order to apply Lemma \ref{lem:John} we first show that for
    \begin{align}
        \label{def:I}
        \cI_t^n:=Z_t-Z_t^n=Y_t-Y_t^n,
    \end{align}
 for any  small $\kappa>0$ so that $\kappa N^*\leq \frac{1}{8}$ where constant $N^*$ is from \eqref{est:qua} (which was denoted as $N$ therein) and for any $(S,T)\in[0,1]_\leq^2$ with $T\in[S,1\wedge(S+\kappa)]$  we have
    \begin{align}
        \label{est:claim}
        \| \cI_T^n-\cI_S^n\|_{L_2(\Omega)|\cF_S}\leq \kappa^\frac{1}{2}(n^{\rho}+|\cI_S^n|)+4\|b\|_{\infty}n^{-1}.
    \end{align}
    According to the fact if $S$ is a gridpoint or not we divide the proof into two cases.

\begin{itemize}
      \item[Case I.] $S\in \{0,\frac{1}{n},\ldots, 1\}$.
  \end{itemize}

    We fix $S$.  We get that from \eqref{eq:main-error} and \eqref{def:I} that
    \begin{align*}
     \cI_t^n-\cI_s^n=    Y_t-Y_s-(Y_t^n-Y_s^n)= Z_t-Z_s-(Z_t^n-Z_s^n),
    \end{align*}
    together with taking $\frac{k}{n}=S$ in \eqref{est:qua} yields
\begin{align*}
 \| \cI_t^n-\cI_s^n\|_{L_2(\Omega)|\cF_S}= &\| Y_t-Y_s-(Y_t^n-Y_s^n)\|_{L_2(\Omega)|\cF_S}
 \\\leq& \Big\| \left(\begin{smallmatrix}\int_s^tV_r-V_r^n dr\\\int_s^tb(Z_r)-\Gamma_{r-k_n(r)}b(Z_{k_n(r)}^n)dr
    \end{smallmatrix}\right)\Big\|_{L_2(\Omega)|\cF_S}
     \\\leq& 2\| \int_s^t\int_s^rb(Z_u)-\Gamma_{u-k_n(u)}b(Z_{k_n(u)}^n)dudr\|_{L_2(\Omega)|\cF_S} \\&+2\|\int_s^tb(Z_r)-\Gamma_{r-k_n(r)}b(Z_{k_n(r)}^n)dr\|_{L_2(\Omega)|\cF_S}
     \\\leq& 4N^*|t-s|^{\frac{1}{2}+\epsilon}(n^{-\rho}+[Y-Y^n]_{C_p^{\frac{1}{2}}|\cF_{S},[s,t]}+|Y_S-Y_S^n|).
\end{align*}
The above estimate implies, by \eqref{def:norms-random-2} and the fact that  $\kappa N^*\leq \frac{1}{8}$, $T\in[S,1\wedge (S+\kappa)]$ so $t-s\leq T-S\leq \kappa$,  and $Y_S-Y_S^n=\cI_S^n$ by definition of $\cI^n$,
\begin{align*}
    [\cI^n]_{C_p^{\frac{1}{2}}|\cF_{S},[s,t]}\leq
  N^*\kappa(n^{-\rho}+ [\cI^n]_{C_p^{\frac{1}{2}}|\cF_{S},[s,t]}+|\cI_S^n|).
\end{align*}
By the choice at the beginning that $\kappa N^*<1$ we get
\begin{align} \label{est:caseI-grid}
    [\cI^n]_{C_p^{\frac{1}{2}}|\cF_{S},[s,t]}\leq
 n^{-\rho}+ |\cI_S^n|.
\end{align}
    \begin{itemize}
      \item[Case II.] $S\in [0,1]\backslash\{0,\frac{1}{n},\ldots, 1\}$.
  \end{itemize}

  We even simply assume $S\in[0,T)$ for generality to show \eqref{est:claim}.

  If $T-S\leq\frac{1}{n}$, then we can easily see that
  \begin{align}\label{est:S-Shat}
  |\cI_T^n-\cI_S^n|    \leq 2\|b\|_{\infty}|T-S|\leq 2\|b\|_{\infty}n^{-1},
  \end{align}
  which implies that \eqref{est:claim} holds.

  If $T-S>\frac{1}{n}$, denote $\hat S:=k_n(S)+\frac{1}{n}$ so $\hat S > S$, we have
  \begin{align}\label{est:I-iN}
      \| \cI_T^n-\cI_S^n\|_{L_2(\Omega)|\cF_S}\leq&   \| \cI_T^n-\cI_{\hat S}^n\|_{L_2(\Omega)|\cF_S}+  \| \cI_{\hat S}^n-\cI_S^n\|_{L_2(\Omega)|\cF_S}
      \nonumber\\=& \big(\mE(\mE(|\cI_T^n-\cI_{\hat S}^n|^2|\cF_{S})|\cF_{\hat S})\big)^\frac{1}{2}+  \| \cI_{\hat S}^n-\cI_S^n\|_{L_2(\Omega)|\cF_S}.
  \end{align}
  Considering that $\hat S$ is gridpoint and $T\leq \hat S+\kappa$, therefore we can apply the result from Case I \eqref{est:caseI-grid} then get
  \begin{align}\label{est:I-iN-1}
      \| \cI_T^n-\cI_{\hat S}^n\|_{L_2(\Omega)|\cF_{\hat S}}\leq& \kappa^\frac{1}{2}(
 n^{-\rho}+ |\cI_{\hat S}^n|)\leq  \kappa^\frac{1}{2}(
 n^{-\rho}+ |\cI_{ S}^n|)+|\cI_{\hat S}^n-\cI_S^n|.
  \end{align}
  From \eqref{est:S-Shat} we get
  \begin{align}\label{est:I-iN-2}
      |\cI_{\hat S}^n-\cI_S^n|\leq 2\|b\|_{\infty}|\hat S-S|\leq 2\|b\|_{\infty}n^{-1}.
  \end{align}
  Now we put \eqref{est:I-iN}, \eqref{est:I-iN-1} and \eqref{est:I-iN-2} together then get \eqref{est:claim}.\\

  With having \eqref{est:claim} we are ready to apply Lemma \ref{lem:John}.  Set $\cA:=\cI^n$ here and $\xi_t:=\kappa^\frac{1}{2}(n^{-\rho}+|\cA_t|)+4\|b\|_\infty n^{-1}$, $t\in[0,1]$. We see that both processes are continuous and \eqref{con:John} is satisfied because of \eqref{est:claim}. Therefore from \eqref{est:lem-John} we obtain that for any $p\geq1$ there exists constant $N=N(p)$ independent of $n$ so that for any $(S,T)\in[0,1]_\leq^2$ and $T\leq S+\kappa$ the following holds:
  \begin{align*}
    & \big \|\sup_{r\in[S,T]}|\cA_r-\cA_S|\big\|_{L_p(\Omega)|\cF_{ S}}\\\leq &N(\kappa^\frac{1}{2}n^{-\rho}+\kappa^\frac{1}{2} \big \|\sup_{r\in[S,T]}|\cA_r|\big\|_{L_p(\Omega)|\cF_{ S}}+4\|b\|_{\infty}n^{-1})
      \\\leq &N(\kappa^\frac{1}{2}n^{-\rho}+\kappa^\frac{1}{2}|\cA_S|  +\kappa^\frac{1}{2}\big\|\sup_{r\in[S,T]}|\cA_r-\cA_S|\big\|_{L_p(\Omega)|\cF_{ S}}+4\|b\|_{\infty}n^{-1}).
  \end{align*}
  By choosing $\kappa$ to be sufficiently small so that also we have $N\kappa\leq\frac{1}{2}$ since  $\kappa N^*\leq \frac{1}{8}$ initially,
  we  hence get
  \begin{align*}
   \big \|\sup_{r\in[S,T]}|\cA_r-\cA_S|\big\|_{L_p(\Omega)|\cF_{ S}}\leq  n^{-\rho}+|\cA_S|  +8N\|b\|_{\infty}n^{-1}.
  \end{align*}
  We can iterate this procedure for $\lceil \kappa^{-1}\rceil$ times then in the end get there exists $N'=N(d,\alpha,\beta,p,N^*)$ so that
  \begin{align*}
        \big\|\sup_{r\in[S,T]}|\cA_r|\big\|_{L_p(\Omega)}\leq  N'n^{-\rho}=N'n^{-(\frac{1}{2}+\frac{\beta}{\alpha(1+\alpha)}\wedge \frac{1}{2})+\epsilon},
  \end{align*}
which is the desired estimate \eqref{est:thm-EM-s}.  The proof finally completes.
\end{proof}
\begin{proof}[Proof of Corollary \ref{cor:a.s}]
    Take $\epsilon>0$ and let $p>\frac{1}{\epsilon}$. Denote $\Xi(\omega):=\sup_n(   \sup_{t\in[0,1]}|Z_t(\omega)-Z_t^n(\omega)|n^{(\frac{1}{2}+\frac{\beta}{\alpha(1+\alpha)}\wedge \frac{1}{2})-2\epsilon})$. From Theorem \ref{thm:main-rough} we have
    \begin{align*}
        \mE\Xi^p&=\mE \sup_n(   \sup_{t\in[0,1]}|Z_t(\omega)-Z_t^n(\omega)|n^{(\frac{1}{2}+\frac{\beta}{\alpha(1+\alpha)}\wedge \frac{1}{2})-2\epsilon})^p\\&\leq \mE \sum_n(   \sup_{t\in[0,1]}|Z_t(\omega)-Z_t^n(\omega)|n^{-(\frac{1}{2}+\frac{\beta}{\alpha(1+\alpha)}\wedge \frac{1}{2})-2\epsilon})^p
        \\&\leq N\sum_nn^{-p\epsilon}<\infty,
    \end{align*}
    which implies that $\Xi<\infty$ a.s.. The proof completes.
\end{proof}





\begin{thebibliography}{99}
\bibitem{ABM}  S. Athreya, O. Butkovsky and Leonid Mytnik: Strong existence and uniqueness for stable stochastic differential equations with distributional drift. {\it Ann. Probab.}  {\bf 48}, 178 - 210,  2020.
\bibitem{B}F. Bouchut: Hypoelliptic regularity in kinetic equations.  {\it J. Math Pures Appl.} {\bf 81},
1135–1159, 2002.
 \bibitem{BDG} O. Butkovsky, K. Dareiotis and M, Gerencs\'er: Approximation of SDEs -- a stochastic sewing approach. {\it Probab. Theory Relat. Fields.} {\bf 181}, 975-1034, 2021.
 \bibitem{BDG-Levy} O. Butkovsky, K. Dareiotis and M, Gerencs\'er: Strong rate of convergence of the Euler scheme for SDEs with irregular drift driven by L\'evy noise. {\it Ann. Inst. Henri Poincar\'e Probab. Stat.} to appear, 2024.
 \bibitem{C} C. Cercignani: Boltzmann Equation and Its Applications. {\it Springer-Verlag, New York,} 1988.
\bibitem{CMPZ} P. Chaudru de Raynal, S. Menozzi, A. Pesce and X. Zhang:
Heat kernel and gradient estimates for kinetic SDEs with low regularity coefficients,
{\it Bulletin des Sciences Mathématiques,} {\bf 183},  2023.
\bibitem{Csz} Z. Chen, R. Song and X. Zhang: Stochastic flows for Lévy processes with Hölder drifts. {\it Rev. Mat. Iberoam. } {\bf 34},  1755–1788, 2018.
\bibitem{CZZ}Z. Chen X. Zhang and G. Zhao: Well-posedness of supercritical SDE driven by Lévy processes with irregular drifts. {\it Transactions of the American Mathematical Society} {\bf 374}, 7621-7655, 2021.
  \bibitem{DGL} K. Dareiotis,  M, Gerencs\'er and K. L\^e: Quantifying convergence theorem of Gy\"ongy and Krylov, {\it Ann. Appl. Probab.}, {\bf 33}, 2291-2323, 2023.
 \bibitem{DF} F. Delarue and S. Menozzi: Density estimates for a random noise propagating
through a chain of differential equations. {\it Journal of Functional Analysis} {\bf 259},
1577 – 1630, 2010.


  \bibitem{HLL} Z. Hao, K. L\^e and C. Ling: Quantitative approximation of stochastic kinetic equations: from discrete to continuum. {\it https://arxiv.org/pdf/2409.05706},   2024.
 \bibitem{HZZZ22} Z. Hao Z., X. Zhang, R. Zhu R. and X. Zhu: Singular kinetic equations. {\it Ann. Probab.} {\bf 52}, 576-657,  2024.
  \bibitem{HRZ} Z. Hao Z., M. R\"ockner and X. Zhang: Second order fractional mean-field SDEs with singular kernels and measure initial data. {\it Ann. Probab.} to appear, Available in  arXiv:2302.04392,  2024.
\bibitem{HWZ20}
 Z. Hao, M. Wu, and X. Zhang,  {\it Schauder's estimate for nonlocal kinetic equations and applications. J. Math. Pures Appl.} {\bf 9},  139-184, 2020.
 \bibitem{HL18} X. Huang and Z. Liao: The Euler-Maruyama method for S(F)DEs with H\"Older drift and $\alpha$-stable noise. {\it Stoch. Anal. Appl.} {\bf 36}, 28–39, 2018.
 \bibitem{HZ} H. Hou and X. Zhang: Heat kernel estimates for nonlocal kinetic operators. {\it arXiv:2410.18614
}, 2024.

       \bibitem{GL}
L. Galeati and C. Ling.
\newblock { Stability estimates for singular SDEs and applications}.
\newblock {\it  Electron. J. Probab.} {\bf 28}, 1--31, 2023.
\bibitem{GLL} M. Gerenscér, G. Lampl and C. Ling: The Milstein Scheme for singular SDEs with Hölder continuous drift. {\it  IMA J. Numer. Anal.} to appear, 2024.
\bibitem{G}  L. Gr\"afner: Energy solutions to SDEs with supercritical
distributional drift: A stopping argument. {\it https://arxiv.org/pdf/2407.09222}, 2024.
\bibitem{GM} L. Gr\"afner and N. Perkowski: Weak well-posedness of energy solutions to singular SDEs with
supercritical distributional drift. {\it https://arxiv.org/pdf/2407.09046}, 2024.

     \bibitem{IS}  C. Imbert and L. Silvestre: The Schauder estimate for kinetic integral equations. {\it Analysis \& PDE} {\bf 14}, 171-204, 2021.

\bibitem{KP} H. Kremp and N. Perkowski: Rough weak solutions for singular L\'evy SDEs. ArXiv preprint. {\it https://arxiv.org/pdf/2309.15460}, 2023.
\bibitem{KR}   N. V. Krylov  and M. R\"ockner: Strong solutions of stochastic equations with singular time dependent drift.
	{\it Probab. Theory Relat. Fields} {\bf{131}}, 154-196, 2005

  \bibitem{KS19} F. K\"uhn and R. Schilling: Strong convergence of the Euler-Maruyama approximation for a class of L\'evy-driven SDEs. {\it Stochastic Process. Appl.} {\bf 129}, 2654–2680, 2019.



  \bibitem{LM} V. Lemaire and  S. Menozzi: On some Non Asymptotic Bounds for the Euler Scheme. {\it  Electron. J. Probab.}  {\bf 15}, 1645-1681,  2010.
  \bibitem{LS18} G. Leobacher and M. Sz\"olgyenyi: Convergence of the Euler-Maruyama method for multidimensional SDEs with discontinuous drift and degenerate diffusion coeffcient. {\it Numer. Math.},
{\bf 138}, 219–239, 2018.
\bibitem{Le} K. L\^e: A stochastic sewing lemma and applications. {\it  Electron. J. Probab.} {\bf 25},
 1-55, 2021
 \bibitem{Le2022} K. L\^e: Quantitative John--Nirenberg inequality for stochastic processes of bounded mean oscillation. {\it Arxiv preprint} https://arxiv.org/pdf/2210.15736, 2022.
 \bibitem{LL} K. L\^e and C. Ling: Taming singular SDEs:   A numerical method.
  Arxiv preprint    {\it https://arxiv.org/pdf/2110.01343.pdf},
  2021.\bibitem{LZ}Y. Li and G. Zhao: Euler–Maruyama scheme for SDE driven by L\'evy process with Hölder drift. {\it Statistics \& Probability Letters}, {\bf 215}, 2024.
 \bibitem{LZ1} C. Ling and G. Zhao: Nonlocal elliptic equation in Hölder space and the martingale problem. {\it Journal of Differential Equations}  {\bf 314},  653-699, 2022.

 \bibitem{MSH}J. Mattingly, J. Stuart, and D. Higham: Ergodicity for SDEs and approximations:
locally Lipschitz vector fields and degenerate noise. {\it Stoch. Proc. Appl.} {\bf 101}, 185–232, 2002.
\bibitem{MGY20} T. M\"uller-Gronbach and L. Yaroslavtseva: A strong order 3/4 method for SDEs with discontinuous drift coefficient, {\it  IMA J. Numer. Anal.} {\bf 42}, 229–25, 2022.

\bibitem{MX} R. Mikulevi\`cius and F. Xu. On the rate of convergence of strong Euler approximation for SDEs driven
by L\'evy processes. {\it Stochastics}, {\bf 90}, 569–604, 2018.
\bibitem{PM} P. Monmarch\'e: High-dimensional MCMC with a standard splitting
scheme for the underdamped Langevin diffusion. {\it https://arxiv.org/pdf/2007.05455}, 2021.

\bibitem{PS}P. Przybylowicz, M. Sz\"olgyenyi: Existence, uniqueness, and approximation of solutions of jump-diffusion SDEs with dis-continuous drift, {\it Appl. Math. Comput.} {\bf 403}, 126191, 2021.

\bibitem{PSS}P. Przybylowicz, V. Schwarz and M. Sz\"olgyenyi:
A higher order approximation method for jump-diffusion SDEs
with discontinuous drift coefficient. {\it J. Math.Anal.Appl.} {\bf 538}, 2024.
\bibitem{PSS1}P. Przybylowicz, V. Schwarz and M. Sz\"olgyenyi: Lower error bounds and optimality of approximation for jump-diffusion SDEs with discontinuous drift, {\it arXiv :2303 .05945}, 2023.


 \bibitem{P} E. Priola:
Pathwise uniqueness for singular SDEs driven by stable processes. {\it Osaka J. Math.} {\bf 49}, 421–447, 2012.

 \bibitem{RZ24} C. Ren and X. Zhang: Heat kernel estimates for kinetic SDEs with drifts being unbounded and in Kato’s class.  {\it Arxiv preprint} arXiv:2401.13873, 2024

\bibitem{SW} K. Schuh
and P. Whalley: Convergence of kinetic Langevin samplers for non-convex potentials. {\it https://arxiv.org/pdf/2405.09992}, 2024.



\bibitem{Tri06} H. Triebel: Theory of function spaces, III. Basel, Birkh\"{a}user, 2006.

\bibitem{T} D. Talay: Stochastic Hamiltonian Systems: Exponential Convergence to the
Invariant Measure, and Discretization by the Implicit Euler Scheme. {\it Markov Processes Relat. Fields.} {\bf 8}, 1–36, 2002
\bibitem{TTW} H. Tanaka, M. Tsuchiya and S. Watanabe: Perturbation of drift-type for Lévy processes, {\it J. Math.
Kyoto Univ.} {\bf 14} 73–92, 1974.
\bibitem{V} C. Villani: A review of mathematical topics in collisional kinetic theory. {\it Handbook of mathematical fluid dynamics}, 2002.
   \bibitem{Ver} A. Yu. Veretennikov: On the strong solutions of stochastic differential equations. {\it Theory Probab. Appl.} {\bf 24}, 354-366,
 1979.
\bibitem{XZ} L. Xie and X. Zhang: Ergodicity of stochastic differential equations with
jumps and singular coefficients. {\it Ann. Inst. Henri Poincaré Probab. Stat.} {\bf 56},  175–229, 2020.

\bibitem{Zhang2011} X. Zhang: Stochastic homeomorphism flows of SDEs with singular drifts and Sobolev diffusion coefficients. {\it Electron. J. Probab.} {\bf 16}, 1096–1116, 2011.
\bibitem{Zhang2018} X. Zhang: Stochastic Hamiltonian flows with singular coefficients. {\it Science China Mathematics}, {\bf{61}},
1353-1384, 2018.
\bibitem{ZZ21} X. Zhang and X. Zhang: Cauchy problem of stochastic kinetic equations. To appear in {\it Ann. Appl. Probab.} Available in arXiv:2103.02267,  2021.

\bibitem{Zv} A. K.  Zvonkin: A transformation of the phase space of a diffusion process that removes the drift. {\it Math. Sbornik}, {\bf 135},  129-149, 1974.

\end{thebibliography}


 \end{document}